\documentclass{amsart}
\usepackage[margin=2.6cm]{geometry}
\usepackage{graphics, amsmath,amssymb, enumitem, comment, url, mathtools}
\usepackage{graphicx}
\usepackage{epsfig}

\newif\ifcolorcomments
\newcommand{\allowcomments}[4]{
\newcommand{#1}[1]{\ifdraft{\ifcolorcomments{\textcolor{#4}{##1 --#3}}\else{\textsl{ ##1 \ --#3}}\fi}\else{}\fi}
}

\colorcommentstrue
\usepackage{amssymb}
\usepackage{amsmath,amsthm}

\usepackage{colortbl}

\newtheorem{theorem}{Theorem}[section]
\newtheorem{lemma}[theorem]{Lemma}
\newtheorem{proposition}[theorem]{Proposition}
\newtheorem{corollary}[theorem]{Corollary}
\theoremstyle{definition}
\newtheorem{definition}[theorem]{Definition}

\newcommand{\EE}{\mathcal E}

\newcommand{\FF}{\mathcal F}

\newcommand{\HH}{\mathcal H}

\newcommand{\N}{\mathbb N}

\newcommand{\R}{\mathbb R}

\newcommand{\UU}{\mathcal U}

\renewcommand{\text}{\textup}

\newcommand{\NPC}[1]{\ignorespaces}

\renewcommand{\H}{\mathbb H}

\newif\ifdraft\drafttrue

\def\N{\mathbb N}

\def\R{\mathbb R}

\def\H{\mathcal H}
\def\a{\alpha}
\newcommand\hdim{\dim_{\mathrm H}}
\newcommand{\subscript}[2]{$#1 _ #2$}

\IfFileExists{marvosym.sty}{
\RequirePackage{marvosym} 
}{

}

\IfFileExists{wasysym.sty}{
\RequirePackage{wasysym}\renewcommand{\emptyset}{{\diameter}}
}{

}

\newcommand*{\myDots}{\ifmmode\mathellipsis\else.\kern-0.07em.\kern-0.07em.\fi}
\DeclarePairedDelimiter\floor{\lfloor}{\rfloor}

\newcommand*{\defeq}{\stackrel{\text{def}}{=}}

\allowcomments{\commumtaz}{MH}{Mumtaz}{green}
\allowcomments{\comnikita}{NS}{Nikita}{blue}
\allowcomments{\combixuan}{BL}{Bixuan}{red}

\newcommand {\ignore}[1] {}

\newcommand{\commenty}[1]{}

\begin{document}

\title[A Hausdorff dimension analysis of an exceptional set]{A Hausdorff dimension analysis of sets with the product of consecutive vs single partial quotients in continued  fractions}

\author[Mumtaz Hussain]{Mumtaz Hussain}
\address{Mumtaz Hussain,  Department of Mathematical and Physical Sciences,  La Trobe University, Bendigo 3552, Australia. }
\email{m.hussain@latrobe.edu.au}

\author[Bixuan Li]{Bixuan Li}
\address{Bixuan Li, School of Mathematics and Statistics, Huazhong University of Science and Technology, Wuhan 430074, P. R. China}
\email{libixuan@hust.edu.cn}

\author{Nikita Shulga}
\address{Nikita Shulga,  Department of Mathematical and Physical Sciences,  La Trobe University, Bendigo 3552, Australia. }
\email{n.shulga@latrobe.edu.au}
\date{}

\maketitle

\numberwithin{equation}{section}

\begin{abstract}
We present a detailed Hausdorff dimension analysis of  the set of real numbers where the product of consecutive partial quotients in their continued fraction expansion grow at a certain rate but the growth of the single partial quotient is at a different rate.  We consider the set
\begin{equation*}
\FF(\Phi_1,\Phi_2) \defeq \EE(\Phi_1) \backslash \EE(\Phi_2)=\left\{x\in[0,1):
\begin{split} 
 a_n(x)a_{n+1}(x) & \geq\Phi_1(n) \text{\,\, for infinitely many } n\in\N \\
 a_{n+1}(x) & <\Phi_2(n)  \text{\,\, for all sufficiently large } n\in\N
\end{split}
\right\},
\end{equation*}
where $\Phi_i:\N\to(0,\infty)$ are any functions such that $\lim\limits_{n\to\infty} \Phi_i(n)=\infty$. We obtain some surprising results including the situations when $\FF(\Phi_1,\Phi_2)$ is empty for various non-trivial choices of $\Phi_i$'s. Our results contribute to the metrical theory of continued fractions by generalising several known results including the main result of [Nonlinearity, 33(6):2615--2639, 2020]. To obtain some of the results, we consider an alternate generalised set, which may be of independent interest, and calculate its Hausdorff dimension. One of the main ingredients is in the usage of the classical mass distribution principle;  specifically a careful distribution of the mass on the Cantor subset by introducing a new idea of two different types of probability measures.
\end{abstract}

\section{ Introduction}
It is well-known that every irrational number $x\in (0, 1)$ has a unique infinite continued fraction expansion. This expansion can be induced by the Gauss map  $T: [0,1)\to [0,1)$ defined by
\[T(0)=0, ~ T(x)=\frac{1}{x}-\floor*{\frac{1}{x}} \textmd{ for }x\in(0,1),\]
where $\lfloor x\rfloor$ denotes the integer part of $x$.  We write $x:=[a_{1}(x),a_{2}(x),a_{3}(x),\ldots ]$ for the continued fraction of $x$ where $a_1(x)=\lfloor 1/x \rfloor$, $a_{n}(x)= a_1(T^{n-1}(x))$ for $n\ge2$ are called the partial quotients of $x$. 

The metric theory of continued fractions concerns the quantitative study of properties of partial quotients for almost all $x\in(0, 1)$. This area of research is closely connected with metric Diophantine approximation, for example, fundamental theorems in this field by Khintchine (1924) and Jarn\'ik (1931) can be represented in terms of the growth of partial quotients. The classical Borel-Bernstein theorem (1912) states that the Lebesgue measure of the set 
\begin{equation*}
\EE_1(\Phi):=\left\{x\in [0, 1): a_n(x)\geq \Phi(n) \ {\rm for \ infinitely \ many} \ n\in \N\right\}
\end{equation*}
is either zero or one depending upon the convergence or divergence of the series $\sum_{n=1}^\infty \Phi(n)^{-1}$ respectively. Here and throughout  $\Phi:\N\to [1, \infty)$ will be an arbitrary function such that $\Phi(n)\to\infty$ as $n\to \infty$.  For rapidly increasing functions $\Phi$, the Borel-Bernstein (1911, 1912) theorem gives no  further information than the zero Lebesgue measure of $\EE_1(\Phi)$. To distinguish between the sets of zero Lebesgue measure,  the Hausdorff dimension is an appropriate tool. 
 When the function $\Phi(n)$ tends to infinity at a polynomial rate $n^a$, the Hausdorff dimension of $\EE_1(\Phi)$ was calculated by Good \cite{Good}, and when $\Phi(n)$ tends to infinity at a super-exponential rate $a^{b^n},$ the Hausdorff dimension was established by {\L}uczak \cite{Luczak}. For an arbitrary function $\Phi$ the dimension of $\EE_1(\Phi)$ was computed by Wang-Wu \cite{WaWu08}.
\begin{theorem}[{\protect\cite[Wang-Wu]{WaWu08}}]
\label{WaWu}

Let $\Phi :\mathbb{N}\rightarrow \mathbb [1, \infty)$. Suppose $$\log B=\liminf\limits_{n\rightarrow \infty }\frac{\log
\Phi (n)}{n} \ { and}\ \log b=\liminf\limits_{n\rightarrow \infty }\frac{\log
\log \Phi (n)}{n}.$$ 

\begin{itemize}
\item [\rm(i)] When $B=1$, $\dim_{H} \mathcal{E}_{1}(\Phi)=1.$

\item[\rm{(ii)}] When $B=\infty$, $\dim_{H} \mathcal{E}
_{1}(\Phi)=1/(1+b).$

\item[\rm{(iii)}] When $1<B<\infty$, $\dim_{H} \mathcal{E}
_{1}(\Phi)=s_{B}:=\inf \{s\geq 0 :\mathsf{P}(T, -s(\log B+\log |T^{\prime
}|))\le 0\},$ 
\end{itemize}
where $
T^{\prime }$ denotes the derivative of  the Gauss map $T$, and $\mathsf{P}$ represents the pressure function defined in the Subection \ref{Pressure Functions}.
\end{theorem}

Recently, Kleinbock-Wadleigh \cite{KleinbockWadleigh} observed that the growth of the product of consecutive partial quotients gives a characterisation for a real number to be $\Phi$-Dirichlet improvable or not. In fact, in \cite{KleinbockWadleigh} the Lebesgue measure of the set of $\Phi$-Dirichlet non-improvable numbers
\begin{equation*}
\EE_2(\Phi):=\left\{x\in [0, 1): a_n(x)a_{n+1}(x)\geq \Phi(n) \ {\rm for \ infinitely \ many} \ n\in \N\right\}
\end{equation*}
was established. Soon after that the Hausdorff dimension of this set (in a generalised form) was established in \cite{HuWuXu, HKWW}.

\begin{theorem}[{\protect\cite[Huang-Wu-Xu]{HuWuXu}}]
\label{HWXthm}
Let  $\Phi :\mathbb{N}\rightarrow \mathbb [1, \infty)$.  Suppose $$\log B=\liminf\limits_{n\rightarrow \infty }\frac{\log
\Phi (n)}{n} \ { and}\ \log b=\liminf\limits_{n\rightarrow \infty }\frac{\log
\log \Phi (n)}{n}.$$ 

\begin{itemize}
\item [\rm(i)] When $B=1$, $\dim_{H} \mathcal{E}_{2}(\Phi)=1.$

\item[\rm{(ii)}] When $B=\infty$, $\dim_{H} \mathcal{E}
_{2}(\Phi)=1/(1+b).$

\item[\rm{(iii)}] When $1<B<\infty$, $\dim_{H} \mathcal{E}
_{2}(\Phi)=s_{0}:=\inf \{s\geq 0 :\mathsf{P}(T, -s^2\log B-s\log |T^{\prime
}|))\le 0\}.$ 
\end{itemize}
\end{theorem}

Note that the set $\mathcal{E}_{1}(\Phi)$ is properly contained in $\mathcal{E}_{2}(\Phi)$ raising a natural question of the size of the set  
\begin{equation*}\label{pqc}
 \mathcal{F}(\Phi ):=~\mathcal{E}_{2}(\Phi )\setminus \mathcal{E}_{1}(\Phi )=\left\{ x\in [
0,1):
\begin{array}{r}
a_{n+1}(x)a_{n}(x)\geq \Phi (n)\text{ for infinitely many }n\in 
\mathbb{N}
\text{ and} \\ 
a_{n+1}(x)<\Phi (n)\text{ for all sufficiently large }n\in 
\mathbb{N}
\end{array}
\right\}.  
\end{equation*} 
 The Hausdorff dimension of $ \mathcal{F}(\Phi )$ was established in \cite{BBH2}.
\begin{theorem}[{\protect\cite[Bakhtawar-Bos-Hussain]{BBH2}}] \label{BBHthm} Let  $\Phi :\mathbb{N}\rightarrow \mathbb [1, \infty)$.  Suppose $$\log B=\liminf\limits_{n\rightarrow \infty }\frac{\log
\Phi (n)}{n} \ {\rm and} \ \log b=\liminf\limits_{n\rightarrow \infty }\frac{\log
\log \Phi (n)}{n}.$$ Then
\begin{equation*}
\dim _{H}\mathcal{F}(\Phi )=\left\{ 
\begin{array}{ll}
s_0 & 
\mathrm{if}\ \ 1<B< \infty; \\ [3ex] \frac{1}{1+b} & 
\mathrm{if}\ \ B=\infty,
\end{array}
\right.  
\end{equation*}
where $s_0$ is from Theorem \ref{HWXthm}.\end{theorem}



Note that the Hausdorff dimension proved in Theorem \ref{BBHthm} is the same as in Theorem \ref{HWXthm}, that is, $\dim_H \mathcal{F}(\Phi )=\dim_H \EE_2(\Phi )$. The main reason that the subset $\mathcal{F}(\Phi )$ has the same dimension as the superset $\EE_{2}(\Phi)$, is that the set $\EE_{2}(\Phi)$ is much larger than $\EE_{1}(\Phi)$.  However, if we consider different approximation functions then the problem becomes more intricate and challenging. The complications arises in constructing a suitable Cantor type subset supporting a probability measure. To overcome this difficulty we introduce two types of probability measures --  an idea that we believe can lend itself in resolving problems of similar nature. Indeed our results below highlight situations when $\mathcal{F}(\Phi )$ and $\EE_{2}(\Phi)$ have different dimensions for different $\Phi$. To be precise,  define the set
\begin{equation*}
\FF(\Phi_1,\Phi_2) \defeq \EE_2(\Phi_1) \backslash \EE_1(\Phi_2)=\left\{x\in[0,1):
\begin{split} 
 a_n(x)a_{n+1}(x) & \geq\Phi_1(n) \text{\,\, for infinitely many } n\in\N \\
 a_{n+1}(x) & <\Phi_2(n)  \text{\,\, for all sufficiently large } n\in\N
\end{split}
\right\}.
\end{equation*}
We fully characterise the Hausdorff dimension of $\FF(\Phi_1,\Phi_2)$ as well as prove some more results regarding general sets, in some way connected to $\FF(\Phi_1,\Phi_2)$, which are also of their own independent interest. 
 
 It is worth noting that when $\Phi$ is a function of the denominator of the $n^{\rm th}$ convergent $q_n$, the Hausdorff dimension (and measure) study is carried out in a sequence of papers
 \cite{BBH1, BHS, BBJ}.  In particular, these papers give insight into comparing the uniform Diophantine approximation (Dirichlet non-improvable numbers) to the asymptotic Diophantine approximation (Jarn\'ik set).  
\subsection{Main results} 
The first result of this paper gives a full description of the Hausdorff dimension of~$\FF(\Phi_1,\Phi_2)$. 
\begin{theorem}\label{general}
Let $\Phi_1$ and $\Phi_2$ be positive functions such that
\begin{align*}
\log B_1 = \liminf\limits_{n\to\infty} \frac{\log \Phi_1(n)}{n} & \text{ \,\,\,\,\,\, and \,\,\,\,\,} \log b_1 = \liminf\limits_{n\to\infty} \frac{\log\log \Phi_1(n)}{n}.\\
\log B_2 = \lim\limits_{n\to\infty} \frac{\log \Phi_2(n)}{n}  &\text{ \,\,\,\,\,\, and \,\,\,\,\,} \log b_2 = \lim\limits_{n\to\infty} \frac{\log\log \Phi_2(n)}{n}.
\end{align*}

Then

%
%

\begin{equation*}
\dim_H \FF(\Phi_1,\Phi_2) =\left\{ 
\begin{array}{ll}
s_0 & {\rm if}\ \ B_1^{s_0} \le B_2 \ {\rm and} \ B_1 \ {\rm is \ finite}; \\ [3ex] 
g_{B_1,B_2} & {\rm if} \ \  B_1^{s_0} \ge B_2 > B_1^{1/2};\\ [3ex] 
\emptyset& {\rm if} \ \  B_1^{1/2}> B_2 \ {\rm and } \ B_2 \ {\rm is \ finite};
\\ [3ex] 
\frac1{1+b_1} & {\rm if} \ \  B_1=B_2=\infty, \ 1\le b_1<b_2, \  {\rm and} \ b_1 \ {\rm is \ finite};
\\ [3ex] 
\emptyset & {\rm if} \ \  B_1=B_2=\infty, \ b_1>b_2\geq 1, \  {\rm and} \ b_2 \ {\rm is \ finite};
\\ [3ex] 
 \emptyset  \text{\,\,  or\,\, }   0 & {\rm if} \ \  B_1=B_2=\infty, \ b_1=b_2=\infty,

\end{array}
\right.
\end{equation*}
where $s_0$ is from Theorem \ref{WaWu} and 
$$g_{B_1,B_2} \defeq \inf \{ s\geq0:P(T,-s\log |T'| -s\log B_1 +(1-s)\log B_2) \leq 0\}.$$
\end{theorem}
The reason behind the slight variation in the definitions of $B_2$ and $b_2$ compared to $B_1$ and $b_1$ respectively,  is in the nature of the conditions of our set. To get the proper classification, we need to bound $\Phi_2(n)$ from both sides for all $n$, while bounding $\Phi_1(n)$ from above for infinitely many $n$ is sufficient. Reader might also notice that this theorem does not cover two cases, namely $B_1 = B_2^2<\infty$ and $b_1=b_2<\infty$. Their omission qualifies them to be highlighted as standalone propositions.

\begin{proposition}\label{Prop1}
Denote
$$
E_1 =\left\{x\in[0,1):
\begin{split} 
 a_n(x)a_{n+1}(x) & \geq B_2^{2n} \text{\,\, for infinitely many } n\in\N \\
 a_{n+1}(x) & < B_2^n  \text{\,\, for all sufficiently large } n\in\N
\end{split}
\right\},
$$
and 
$$
E_2 =\left\{x\in[0,1):
\begin{split} 
 a_n(x)a_{n+1}(x) & \geq B_2^{2n-1} \text{\,\, for infinitely many } n\in\N \\
 a_{n+1}(x) & <3 B_2^{n+1}  \text{\,\, for all sufficiently large } n\in\N
\end{split}
\right\}.
$$
Then $B_1=B_2^2<\infty$ for both of these sets, but
$$
E_1 = \emptyset, \,\,\, \hdim E_2 = g_{B_1,B_2}.
$$
\end{proposition}
\noindent
\begin{proposition}\label{Prop2}
Denote
$$
P_1 =\left\{x\in[0,1):
\begin{split} 
 a_n(x)a_{n+1}(x) & \geq e^{b_1^{n}} \text{\,\, for infinitely many } n\in\N \\
 a_{n+1}(x) & < e^{b_1^{n}}  \text{\,\, for all sufficiently large } n\in\N
\end{split}
\right\},
$$
and
$$
P_2 =\left\{x\in[0,1):
\begin{split} 
 a_n(x)a_{n+1}(x) & \geq e^{5b_1^{n}} \text{\,\, for infinitely many } n\in\N \\
 a_{n+1}(x) & <  e^{b_1^{n-1}}  \text{\,\, for all sufficiently large } n\in\N
\end{split}
\right\}.
$$
Then $B_1=B_2=\infty$, $b_1=b_2<\infty$ for both of these sets, but
$$
P_2 = \emptyset, \,\,\, \hdim P_1 = \frac1{1+b_1}.
$$
\end{proposition}
\medskip

The key ingredient in proving Theorem \ref{general} is to consider a supporting set.  Define
$$
E(A_1,A_2) \defeq \left\{ x\in[0,1): \, c_1 A_1^n\leq a_n(x) <2c_1A_1^n,  \, \,  c_2 A_2^n\leq a_{n+1}(x)<2  c_2 A_2^n,\,\, \text{for i.m. } n\in\N \right\}.
$$
Here and throughout, `i.m.' stands for `infinitely many'. Then we have
\begin{theorem}\label{dim E(A_1,A_2)}
For any $A_1>1$,
$$\dim_H E(A_1,A_2) = \min \left\{ s_{A_1} , g_{(A_1 A_2),A_1} \right \},$$
where $s_{A_1}$ is from Theorem \ref{WaWu} and $g_{(A_1 A_2),A_1}$ was defined in Theorem \ref{general}.
\end{theorem}
\medskip
There is one more supporting set we need to consider to calculate the Hausdorff dimension of $\FF(\Phi_1,\Phi_2)$. Define
\begin{equation*}
\FF_{B_1,B_2}=\left\{x\in[0,1):
\begin{split} 
 a_n(x)a_{n+1}(x) & \geq B_1^n \text{\,\, for infinitely many } n\in\N \\
 a_{n+1}(x) & < B_2^n  \text{\,\, for all sufficiently large } n\in\N
\end{split}
\right\}.
\end{equation*}

This set is a special case of $\FF(\Phi_1,\Phi_2)$ with $\Phi_i(n)=B_i^n,i=1,2$. However, dealing with this set is crucial for the general case. We prove the following theorem explicitly.
\begin{theorem}\label{specific}
For any $B_1,B_2>1$,
 \begin{itemize}
\item when $B_1^{s_0} \le B_2$,
$$ \dim_H \FF_{B_1,B_2} = s_{0};$$
\item when  $B_1^{s_0} \ge B_2 > B_1^{1/2}$,
$$\dim_H  \FF_{B_1,B_2} = g_{B_1,B_2};$$
\item when $B_1^{1/2}\ge B_2$,
$$ \FF_{B_1,B_2}=\emptyset.$$

\end{itemize}
\end{theorem}

The paper is structured as follows. In Section \ref{S2} we provide some auxiliary results and group together various definitions that we shall appeal to in the course of proving our theorems in the subsequent sections. In Section \ref{cantor}, which serves as the main subsection of this paper, we prove Theorem \ref{dim E(A_1,A_2)}, that is,  we find the Hausdorff dimension for the set $E(A_1,A_2)$. In Section \ref{specificsec} we use this set to prove Theorem \ref{specific}. In Section \ref{FPhi1Phi2}, we prove Theorem \ref{general} using Theorem \ref{specific} as well as some auxiliary results. In Section \ref{examples}, we prove both Propositions \ref{Prop1} and \ref{Prop2}.

\medskip

\noindent {\bf Acknowledgements.}   The research of Mumtaz Hussain and Nikita Shulga is supported by the Australian Research Council Discovery Project (200100994). The authors thank Professor Baowei Wang for useful discussions.

\section{Preliminaries and auxiliary results}\label{S2}

For completeness we give a brief introduction to Hausdorff measures and dimension.  For further details we refer to the beautiful texts \cite{BernikDodson, Falconer_book}.

\subsection{Hausdorff measure and dimension}\label{HM}\

Let $0<s\in\R^n$ let
$E\subset \R^n$.
 Then, for any $\rho>0$ a countable collection $\{B_i\}$ of balls in
$\R^n$ with diameters $\mathrm{diam} (B_i)\le \rho$ such that
$E\subset \bigcup_i B_i$ is called a $\rho$-cover of $E$.
Let
\[
\H_\rho^s(E)=\inf \sum_i \mathrm{diam}(B_i)^s,
\]
where the infimum is taken over all possible $\rho$-covers $\{B_i\}$ of $E$. It is easy to see that $\H_\rho^s(E)$ increases as $\rho$ decreases and so approaches a limit as $\rho \rightarrow 0$. This limit could be zero or infinity, or take a finite positive value. Accordingly, the \textit{$s$-Hausdorff measure $\H^s$} of $E$ is defined to be
\[
\H^s(E)=\lim_{\rho\to 0}\H_\rho^s(E).
\]

It is easy to verify that Hausdorff measure is monotonic and countably sub-additive, and that $\H^s(\varnothing)=0$. Thus it is an outer measure on $\R^n$.


For any subset $E$ one can verify that there exists a unique critical value of $s$ at which $\H^s(E)$ `jumps' from infinity to zero. The value taken by $s$ at this discontinuity is referred to as the \textit{Hausdorff dimension of  $E$} and  is denoted by $\hdim E $; i.e.,
\[
\hdim E :=\inf\{s\in \R_+\;:\; \H^s(E)=0\}.
\] When $s=n$,  $\H^n$ coincides with standard Lebesgue measure on $\R^n$.

Computing Hausdorff dimension of a set is typically accomplished in two steps: obtaining the upper and lower bounds separately.
Upper bounds often can be handled by finding appropriate coverings. When dealing with a limsup set, one 
 {usually applies} the Hausdorff measure version of the famous Borel--Cantelli lemma (see Lemma 3.10 of \cite{BernikDodson}):

\begin{proposition}\label{bclem}
    Let $\{B_i\}_{i\ge 1}$ be a sequence of measurable  sets in $\R$ and suppose that,  $$\sum_i \mathrm{diam}(B_i)^s \, < \, \infty.$$ Then  $$\H^s({\limsup_{i\to\infty}B_i})=0.$$
\end{proposition}
The main tool in establishing the lower bound for the dimension of $E(A_1,A_2)$ will be the following well-known mass distribution principle \cite{Falconer_book}.
\begin{proposition}[Mass Distribution Principle \cite{Falconer_book}]\label{p1}
Let $\mu$ be a probability measure supported on a measurable set $F$. Suppose there are positive constants $c$ and $r_0$ such that
$$\mu(B(x,r))\le c r^s$$
for any ball $B(x,r)$ with radius $r\le r_0$ and center $x\in F$. Then $\hdim F\ge s$.
\end{proposition}

\subsection{Continued fractions and Diophantine approximation}
 Recall that the Gauss map $T: [0,1)\to [0,1)$ is defined by
\[T(0)=0, ~ T(x)=\frac{1}{x}-\floor*{\frac{1}{x}} \textmd{ for }x\in(0,1),\]
where $\lfloor x\rfloor$ denotes the integer part of $x$. 
We write $x:=[a_{1}(x),a_{2}(x),a_{3}(x),\ldots ]$ for the continued fraction of $x$ where $a_1(x)=\lfloor 1/x \rfloor$, $a_{n}(x)= a_1(T^{n-1}(x))$ for $n\ge2$ are called the partial quotients of $x$. 
The sequences $p_n= p_n(x)$, $q_n= q_n(x)$, referred to as $n^{\rm th}$ convergents, has the recursive relation
\begin {equation}\label{recu}
p_{n+1}=a_{n+1}(x)p_n+p_{n-1}, \ \
q_{n+1}=a_{n+1}(x)q_n+q_{n-1},\ \  n\geq 0.
\end {equation}
Thus $p_n=p_n(x), q_n=q_n(x)$ are determined by the partial quotients $a_1,\dots,a_n$, so we may write $p_n=p_n(a_1,\dots, a_n), q_n=q_n(a_1,\dots,a_n)$. When it is clear which partial quotients are involved, we denote them by $p_n, q_n$ for simplicity.

For any integer vector $(a_1,\dots,a_n)\in \N^n$ with $n\geq 1$, write
\begin{equation*}\label{cyl}
I_n(a_1,\dots,a_n):=\left\{x\in [0, 1): a_1(x)=a_1, \dots, a_n(x)=a_n\right\}
\end{equation*}
for the corresponding `cylinder of order $n$', i.e.\  the set of all real numbers in $[0,1)$ whose continued fraction expansions begin with $(a_1, \dots, a_n).$

We will frequently use the following well known properties of continued fraction expansions.  They are explained in the standard texts \cite{IosKra_book, Khi_63}.

\begin{proposition}\label{pp3} For any {positive} integers $a_1,\dots,a_n$, let $p_n=p_n(a_1,\dots,a_n)$ and $q_n=q_n(a_1,\dots,a_n)$ be defined recursively by \eqref{recu}. {Then:}
\begin{enumerate}[label={\rm (\subscript{\rm P}{\arabic*})}]
\item
\begin{eqnarray*}
I_n(a_1,a_2,\dots,a_n)= \left\{
\begin{array}{ll}
         \left[\frac{p_n}{q_n}, \frac{p_n+p_{n-1}}{q_n+q_{n-1}}\right)     & {\rm if }\ \
         n\ {\rm{is\ even}};\\
         \left(\frac{p_n+p_{n-1}}{q_n+q_{n-1}}, \frac{p_n}{q_n}\right]     & {\rm if }\ \
         n\ {\rm{is\ odd}}.
\end{array}
        \right.
\end{eqnarray*}
{\rm Thus, its length is given by}
\begin{equation*}\label{lencyl}
\frac{1}{2q_n^2}\leq |I_n(a_1,\ldots,a_n)|=\frac{1}{q_n(q_n+q_{n-1})}\leq \frac1{q_n^2},
\end{equation*}
{\rm since} $$
 p_{n-1}q_n-p_nq_{n-1}=(-1)^n, \ {\rm for \ all }\ n\ge 1.
 $$

\item For any $n\geq 1$, $q_n\geq 2^{(n-1)/2}$ and
$$
1\le \frac{q_{n+m}(a_1,\cdots,a_n, b_1,\cdots, b_m)}{q_n(a_1,\cdots, a_n)\cdot q_m(b_1,\cdots,b_m)}\le 2.
$$
\item $$\prod_{i=1}^na_i\leq q_n\leq \prod_{i=1}^n(a_i+1)\leq 2^n\prod_{i=1}^na_i.$$

\item \begin{equation*}\label{p7}
\frac{1}{3a_{n+1}(x)q^2_n(x)}\, <\, \Big|x-\frac{p_n(x)}{q_n(x)}\Big|=\frac{1}{q_n(x)(q_{n+1}(x)+T^{n+1}(x)  q_n(x))}\, < \,\frac{1}{a_{n+1}q^2_n(x)}.
\end{equation*}
\item there exists a constant $K>1$ such that for almost all $x\in [0,1)$, $$
q_n(x)\le K^n, \ {\text{for all $n$ sufficiently large}}.
$$
\end{enumerate}
\end{proposition}

Let $\mu$ be the Gauss measure given by $$
d\mu=\frac{1}{(1+x)\log 2}dx.
$$ It is clear that $\mu$ is $T$-invariant and equivalent to Lebesgue measure $\mathcal{L}$.

The next proposition concerns the position of a cylinder in $[0,1)$.
\begin{proposition}[\cite{Khi_63}, Khintchine]\label{pp2} Let $I_n=I_n(a_1,\dots, a_n)$ be a cylinder of order $n$, which is partitioned into sub-cylinders $\{I_{n+1}(a_1,\dots,a_n, a_{n+1}): a_{n+1}\in \N\}$. When $n$ is odd, these sub-cylinders are positioned from left to right, as $a_{n+1}$ increases from 1 to $\infty$; when $n$ is even, they are positioned from right to left.
\end{proposition}

The following result is due to {\L}uczak \cite{Luczak}.
\begin{lemma}[\cite{Luczak}, {\L}uczak]\label{lemb}For any $b, c>1$, the sets
\begin{align*}
&\left\{x\in[0, 1):  a_{n}(x)\ge c^{b^n}\  {\text{for infinitely many}} \ n\in \N       \right\},\\
&\left\{x\in[0, 1):  a_{n}(x)\ge c^{b^n}\  {\text{for all }} \ n\geq 1 \right\},
\end{align*}
have the same Hausdorff dimension $\frac1{b+1}$.
\end{lemma}
We will also need the following lemma from \cite[Lemma 3.2]{FLWW}.
\begin{lemma}[\cite{FLWW}, Fan-Liao-Wang-Wu]\label{helplemma}
 Let $\{s_n\}_{n\ge1}$ be a sequence of positive integers tending to infinity with $s_n\ge 3$ for all $n \ge 1$. Then for any positive number $N \ge 2$, we have
$$
\hdim \{x\in[0,1):s_n \leq a_n(x) < N s_n \,\, \forall n\ge 1 \} = \liminf_{n\to\infty} \frac{\log(s_1 s_2\cdots s_n) }{2\log(s_1 s_2\cdots s_n)+\log s_{n+1}}.
$$
\end{lemma}
\subsection{Pressure function}\label{Pressure Functions}
When dealing with dimension problems in non-linear dynamical system, pressure function and other concepts from thermodynamics are good tools. The concept of a general pressure function was introduced by Walters \cite{Walters_book} as a generalization of entropy. It describes exponential growth rate of ergodic sum which determine the whole system in some sense. We are interested in a way of calculating Hausdorff dimensions using pressure functions.

A method in \cite[Theorem 2.2.1]{Pollicott2005} can be used to calculate the Hausdorff dimension of self-similar sets for linear system. As for the non-linear setting, the relation between Hausdorff dimension and pressure functions is given in \cite{Pollicott2005} as the corresponding generalization of Moran \cite{Moran}. The main required ingredient is the following

\begin{definition}
Given any continuous function $f:X\rightarrow 
\mathbb{R}
$ we define its \emph{pressure} $\mathsf P\left( f\right) $, with respect to $T$,  as
\begin{equation*}
\mathsf P\left( f\right) :=\underset{n\rightarrow \infty }{\lim \sup }\frac{1}{n}
\log { {\left(
\sum\limits_{\substack{ T^{n}x=x  \\ x\in X}}e^{f\left( x\right) +f\left(
Tx\right) +\cdots +f\left( T^{n-1}x\right) }\right) }}
\end{equation*}
where the summation is over all periodic points. 
\end{definition}
We note that we can write $``\lim"$ instead of $``\limsup"$ since the above limit actually exists. Combined this with dimension theory, we are concerned with a family of functions $f_{t}\left( x\right) =-t\log
\left\vert T^{\prime }\left( x\right) \right\vert $, $x\in X$ and $0\leq t\leq d$. Substituting these functions into the above definition leads to a real-valued map
\begin{equation*}
\begin{array}{l}
\left[ 0,d\right] \rightarrow 
\mathbb{R},  \qquad
t\mapsto \mathsf P\left( f_{t}\right) =\underset{n\rightarrow \infty }{\lim \sup }
\frac{1}{n}\log \left( \sum\limits_{\substack{ T^{n}x=x  \\ x\in X}}\frac{1
}{\left\vert \left( T^{n}\right) ^{\prime }\left( x\right) \right\vert ^{t}}
\right).
\end{array}
\end{equation*}

We observe that a function $t\mapsto \mathsf P\left( f_{t}\right) $ satisfies interesting properties:

\begin{enumerate}
\item[(i)] $\mathsf P\left( 0\right) =\log k$;

\item[(ii)] $t\mapsto \mathsf P\left( f_{t}\right) $ is strictly 
decreasing; 

\item[(iii)] $t\mapsto \mathsf P\left( f_{t}\right) $ is analytic on $\left[ 0,d
\right] $.
\end{enumerate}

Proof can be found on page 32 of \cite{Pollicott2005}. Moreover, property (i) is easily obtained by the definition.

Bowen and Ruelle had given key results for calculating the Hausdorff dimension from the  pressure function. Bowen
showed the result in the context of quasi-circles,  and Bowen-Ruelle developed the
method for the case of hyperbolic Julia sets. More precisely,

\begin{theorem}[Bowen-Ruelle]
\label{Bowen-Ruelle Theorem}Let $T:X\rightarrow X$ be a $C^{1+\alpha }$
conformal expanding map, for some $\alpha>0$. There is a unique solution $0\leq s\leq d$ to
\begin{equation*}
\mathsf P\left( -s\log \left\vert T^{\prime }\right\vert \right) =0,
\end{equation*}
which occurs precisely at $s=\dim _{H}X$.
\end{theorem}

More details and context on pressure functions can be found in \cite{MaUr96,MaUr99,MaUr03}. We use the fact that the pressure function with a continuous potential can be approximated by the pressure function restricted to the sub-systems in continued fractions.

Let $\mathcal{A}\subset\mathbb{N}$ be a finite or infinite set
and define 
\begin{equation*}
X_{\mathcal{A}}=\{x\in [0,1):{\text{for all}}\ n\geq 1,a_{n}(x)\in 
\mathcal{A}\}.
\end{equation*}
Then $(X_{\mathcal{A}},T)$ is a subsystem of $([0,1),T)$ where $T$ is a
Gauss map. Given any real function $\psi:[0,1)\rightarrow \mathbb{R},$ the pressure function restricted to
the system $(X_{\mathcal{A}},T)$ is defined by 
\begin{equation}\label{5}
\mathsf{P}_{\mathcal{A}}(T,\psi):=\lim_{n\rightarrow \infty }\frac{1}{n}\log \sum_{a_{1},\cdots ,a_{n}\in \mathcal{A}}\sup_{x\in X_{\mathcal{A}}}e^{S_{n}\psi([a_{1},\cdots ,a_{n}+x])},  
\end{equation}
where $S_{n}\psi (x)$ denotes the ergodic sum $\psi(x)+\cdots
+\psi (T^{n-1}x)$. For simplicity, we denote $\mathsf{P}_{\mathbb{N}}(T,\psi)$ by $
\mathsf{P}(T,\psi )$ when $\mathcal{A}=\mathbb{N}$. We note that the supremum in equation \eqref{5} can be removed if $\psi$ satisfy the continuity property.

For each $n\geq1$, the $n^{\rm th}$ variation of $\psi$ is denoted by 
\begin{equation*}
{\mathop{\rm{Var}}}_{n}(\psi):=\sup\Big\{|\psi(x)-\psi(y)|:
I_n(x)=I_n(y)\Big\}.
\end{equation*}
The following results \cite[Proposition 2.4]{LiWaWuXu14} ensure the existence of the limit in the equation (\ref{5}).

\begin{proposition}[{\protect\cite{LiWaWuXu14}, Li-Wang-Wu-Xu}]
\label{pp1} Let $\psi:[0,1)\to \mathbb{R}$ be a real function with $
\mathrm{Var}_1(\psi)<\infty$ and $\mathrm{Var}_{n}(\psi)\to 0$ as $
n\to \infty$. Then the limit defining $\mathsf{P}_\mathcal{A}(T,\psi)$
exists and the value of $\mathsf{P}_\mathcal{A}(T,\psi)$ remains the same
even without taking supremum over $x\in X_\mathcal{A}$ in \eqref{5}.
\end{proposition}

The system $([0,1),T)$ is approximated by its subsystems $(X_{\mathcal{A}},T)
$ then the pressure function has a continuity property in the system of
continued fractions. A detailed proof can be seen in \cite[Proposition 2]{HaMaUr} or \cite{LiWaWuXu14}.

\begin{proposition}[{\protect\cite{HaMaUr}, Hanus-Mauldin-Urba\'nski}]
\label{l5} Let $\psi:[0,1)\to \mathbb{R}$ be a real function with $
\mathrm{Var}_1(\psi)<\infty$ and $\mathrm{Var}_{n}(\psi)\to 0$ as $n\to
\infty$. We have 
\begin{equation*}
\mathsf{P}_{\mathbb{N}}(T, \psi)=\sup\{\mathsf{P}_\mathcal{A}(T,\psi): 
\mathcal{A}\ \mathrm{is \ a \ finite\ subset\ of }\ \mathbb{N}\}.
\end{equation*}
\end{proposition}
Now we consider the specific potentials,
\begin{align*}
\psi_{1}(x)&=-s(\log B+\log |T^{\prime }(x)|)\\ 
\psi_{2}(x)&=-s\log|T^{\prime }(x)|-s\log B_{1}+(1-s)\log B_{2}
\end{align*}
where $
1<B,B_{1},B_{2}<\infty,$ and $s\geq0$.
 It is clear that $\psi_{1}$ and $\psi_{2}$ satisfy the variation condition and then Proposition \ref{l5} holds.

Thus, the pressure function \eqref{5} with potential $\psi_{1}$ is represented by 
\begin{align*}  \label{gddd}
\mathsf P_{\mathcal{A}}(T, -s(\log B+\log |T^{\prime }(x)|) )&=\lim_{n\to\infty}
\frac1n \log \sum_{a_1,\ldots,a_n \in \mathcal{A}} e^{ S_n (-s(\log B+\log
|T^{\prime }(x)|)) }  \notag \\
&=\lim_{n\to\infty} \frac1n \log
\sum_{a_1,\ldots,a_n \in \mathcal{A}}\left(\frac1{B^{n}q_n^2}\right)^s.
\end{align*}
The last equality holds by
\begin{equation*}
{S_n (-s(\log B+\log |T^{\prime }(x)|)) }={-ns\log B-s\log q_{n}^{2}}.
\end{equation*}
which is easy to check by Proposition \ref{pp3}.

As before, we obtain the pressure function with potential $\psi_{2}$ by
\begin{align*}
\mathsf P_{\mathcal{A}}(T, -s\log|T^{\prime }(x)|-s\log B_{1}+(1-s)\log B_{2})&=\lim_{n\to\infty}
\frac1n \log \sum_{a_1,\ldots,a_n \in \mathcal{A}} e^{ S_n (-s\log|T^{\prime }(x)|-s\log B_{1}+(1-s)\log B_{2}) }  \notag \\
&=\lim_{n\to\infty} \frac1n \log
\sum_{a_1,\ldots,a_n \in \mathcal{A}}\left(\frac1{B_{1}^{n}q_n^2}\right)^sB_{2}^{(1-s)n}.
\end{align*}
For any $n\geq1$ and $s\geq0,$ we write
\begin{align*}\label{gnro}
f_{n}^{(1)}\left( s \right)&:=\sum_{a_{1},\ldots ,a_{n}\in \mathcal{A}}\frac{1}{
\left( B^{n}q_{n}^{2}\right) ^{s }}  \\
f_{n}^{(2)}\left( s \right)&:=\sum_{a_1,\ldots,a_n \in \mathcal{A}}\left(\frac1{B_{1}^{n}q_n^2}\right)^sB_{2}^{(1-s)n}.
\end{align*}
and denote
\begin{equation*}
s_{n,B}\left( \mathcal{A}\right) =\inf \left\{ s \geq 0:f_{n}^{(1)}\left( s
\right) \leq 1\right\},g_{n,B_{1},B_{2}}\left( \mathcal{A}\right) =\inf \left\{ s \geq 0:f_{n}^{(2)}\left( s
\right) \leq 1\right\}, \label{tnB}
\end{equation*} 
\begin{align*}  \label{eqpf2}
s_{B}(\mathcal{A})=\inf \{s\geq 0 :\mathsf{P}_{\mathcal{A}}(T, \psi_{1})\le 0\},&g_{B_{1},B_{2}}(\mathcal{A})=\inf \{s\geq 0 :\mathsf{P}_{\mathcal{A}}(T, \psi_{2})\le 0\}, \\
s_{B}(\mathbb{N})=\inf \{s\geq 0 :\mathsf{P}(T, \psi_{1})\le 0\},&g_{B_{1},B_{2}}(\mathbb{N})=\inf \{s\geq 0 :\mathsf{P}(T, \psi_{2})\le 0\}.
\end{align*}

If $\mathcal{A}\in\mathbb{N}$ is a finite set, then by \cite{WaWu08} it is
straightforward to check that both $f_{n}^{(i)}\left( s\right) $ and $\mathsf{P}_{
\mathcal{A}}(T, \psi_{i})$ for $i=1,2$ are monotonically decreasing and
continuous with respect to $s$. Thus, $s_{n,B}\left( \mathcal{A}\right)$, $s_{B}(
\mathcal{A})$, $g_{n,B_{1},B_{2}}\left( \mathcal{A}\right)$ and $g_{B_{1},B_{2}}(
\mathcal{A})$ are, respectively, the unique solutions of $f_{n}^{(1)}\left( s\right)
= 1 $, $\mathsf{P}_{\mathcal{A}}(T, \psi_{1})= 0$, $f_{n}^{(2)}\left( s\right)
= 1 $ and $\mathsf{P}_{\mathcal{A}}(T, \psi_{2})= 0.$

For simplicity, when $\mathcal{A}=\left\{ 1,2,\ldots ,M
\right\}$ for some $M>0$, we write $s_{n,B}\left({M }\right) $ for $
s_{n,B}\left( \mathcal{A}\right) $, $s_{B}\left({M }\right) $ for $
s_{B}\left( \mathcal{A}\right) $, $g_{n,B_{1},B_{2}}\left({M }\right) $ for $
g_{n,B_{1},B_{2}}\left( \mathcal{A}\right) $ and $g_{B_{1},B_{2}}\left({M }\right) $ for $
g_{B_{1},B_{2}}\left( \mathcal{A}\right) $. When $\mathcal{A}=\mathbb{N}$, we write $s_{n,B}$ for $
s_{n,B}(\mathbb{N})$, $s_{B}$ for $
s_{B}(\mathbb{N})$, $g_{n,B_{1},B_{2}}$ for $
g_{n,B_{1},B_{2}}(\mathbb{N})$ and $g_{B_{1},B_{2}}$ for $
g_{B_{1},B_{2}}(\mathbb{N})$.

As a consequence, we have 
\begin{corollary} \label{cor2.1}For any integer $M\in\N,$
\label{p2} 
\begin{equation*}
\lim_{n\to \infty}s_{n,B}(M)=s_{B}(M), \ \ \lim_{M\to \infty}s_{B}(M)=s_{B}, \ \ \lim_{n\to \infty}g_{n,B_{1},B_{2}}(M)=g_{B_{1},B_{2}}(M),\ \ \lim_{M\to \infty}g_{B_{1},B_{2}}(M)=g_{B_{1},B_{2}}.
\end{equation*}
Note that,  $s_{B}$ and $g_{B_{1},B_{2}}$ are continuous respectively as a function of $B$ and $B_{1},B_{2}$. Moreover,
$$\lim_{B\rightarrow 1}s_{B}=1,\ \ \lim_{B\rightarrow\infty}s_{B}=1/2.$$
\end{corollary}
\begin{proof}
The last two equations can be proved by following similar steps as for $s_{B}$ in \cite{WaWu08} and others are consequences of Proposition \ref{l5}.
\end{proof}

\section{Hausdorff dimension of $E(A_1,A_2)$.}\label{cantor}

The proof of Theorem \ref{dim E(A_1,A_2)}  consists of two parts, the upper bound and the lower bound.  
For notational simplicity, we take $c_1=c_2=1$ and the other case can be done with obvious modifications. That is we will deal with the set
$$E(A_1,A_2) \defeq \left\{ x\in[0,1): \,  A_1^n\leq a_n(x) <2A_1^n, \, \, A_2^n \leq a_{n+1}(x)<2  A_2^n,\,\, \text{for i.m. } n\in\N \right\}.$$
\subsection{Upper bound.} It is known that the covering of an upper limit set is natural, so the natural covering of $E(A_1,A_2)$ is found routinely as follows. 
\begin{align*}
E(A_{1})&=\bigcap\limits_{N=1}^{\infty}  \bigcup\limits_{n=N}^{\infty} \left\{ x\in[0,1): \,  A_1^n\leq a_n(x) <2A_1^n, \, \, A_2^n \leq a_{n+1}(x)<2  A_2^n \right\}\\
&\subset\bigcup\limits_{n=N}^{\infty} \left\{ x\in[0,1): \,  A_1^n\leq a_n(x) <2A_1^n, \, \,  A_2^n \leq a_{n+1}(x)<2  A_2^n \right\}:=\bigcup\limits_{n=N}^{\infty} E_{n}\\
&=\bigcup\limits_{n=N}^{\infty}  \bigcup\limits_{a_1,\ldots,a_{n-1}\in\N} \left\{ x\in[0,1)\, : \, a_i(x)=a_i,1\leq i<n, \   A_1^n\leq a_n(x) <2A_1^n, \ A_2^n \leq a_{n+1}(x)<2  A_2^n \right\}.
\end{align*}

There are two potential optimal covers for $E_{n}$ for each $n\ge N$. 
For any integers $a_1,\ldots,a_{n-1}\in\N$, define
$$
J_{n-1}(a_1,\ldots,a_{n-1})=  \bigcup\limits_{A_1^n\leq a_n <2A_1^n } I_n(a_1,\ldots,a_n).
$$
That is one type of suitable coverings of $E(A_{1})$, which is specifically represented as
$$E(A_{1})\subset\bigcup\limits_{n=N}^{\infty}  \bigcup\limits_{a_1,\ldots,a_{n-1}\in\N}J_{n-1}(a_1,\ldots,a_{n-1}).$$
Then, by using Proposition \ref{pp3} and Proposition \ref{pp2} recursively, we obtain
$$
|J_{n-1}(a_1,\ldots,a_{n-1})| =  \sum\limits_{A_1^n\leq a_n <2A_1^n } \left| \frac{p_n}{q_n} - \frac{p_n+p_{n-1}}{q_n+q_{n-1}}\right| \asymp \frac{1}{A_1^n q_{n-1}^2}.
$$
Therefore, for any $\varepsilon>0$, an $(s_{A_{1}}+2\varepsilon)$-dimensional Hausdorff dimension of $E(A_1,A_2)$ can be estimated as
\begin{equation*}
\begin{split}
\HH^{s_{A_{1}}+2\varepsilon}(E(A_1,A_2))& \leq \liminf_{N\to\infty}   \sum\limits_{n=N}^\infty   \sum\limits_{a_1,\ldots,a_{n-1}}  |J_{n-1}(a_1,\ldots,a_{n-1})|^{s_{A_{1}}+2\varepsilon} \\
& \leq  \liminf_{N\to\infty}   \sum\limits_{n=N}^\infty  \frac{1}{2^{(n-1)\varepsilon}} \sum\limits_{a_1,\ldots,a_{n-1}} \left(\frac{1}{A_1^{n} q_{n-1}^{2}}\right)^{s_{A_{1}}}\\
&\le \liminf_{N\to\infty}   \sum\limits_{n=N}^\infty  \frac{1}{2^{(n-1)\varepsilon}}<\infty. 
\end{split}
\end{equation*}
Hence, from the definition of Hausdorff dimension, it follows that 
$$
\dim_H E(A_1,A_2) \leq s_{A_1}.
$$

As for another covering type of $E(A_{1})$, similar with the former, for any integers $a_1,\ldots,a_{n-1}\in\N$ and $A_1^n\leq a_n(x) <2A_1^n$, define
$$
J_{n}(a_1,\ldots,a_{n})=  \bigcup\limits_{A_2^n \leq a_{n+1}<2  A_2^n } I_{n+1} (a_1,\ldots,a_{n+1})
$$
Then
$$
|J_n(a_1,\ldots,a_n) | \asymp  \frac{1}{A_2^n q_{n}^2},
$$
and
$$
E(A_{1}) \subset \bigcup_{n=N}^{\infty} \bigcup\limits_{a_1,\ldots,a_{n-1}}   \bigcup\limits_{ A_1^n\leq a_n <2A_1^n} J_{n}(a_1,\ldots,a_{n}).
$$
Therefore, for any $\varepsilon>0$, a $(g_{(A_1 A_2),A_1}+2\varepsilon)$-dimensional Hausdorff dimension of $E(A_1,A_2)$ can be estimated as
\begin{equation*}
\begin{split}
\HH^{g_{(A_1 A_2),A_1}+2\varepsilon}(E(A_1,A_2))& \leq \liminf_{N\to\infty}   \sum\limits_{n=N}^\infty   \sum\limits_{a_1,\ldots,a_{n-1}}   \sum\limits_{ A_1^n\leq a_n <2A_1^n} |J_{n}(a_1,\ldots,a_{n})|^{g_{(A_1 A_2),A_1}+2\varepsilon} \\
& \leq  \liminf_{N\to\infty}   \sum\limits_{n=N}^\infty   \sum\limits_{a_1,\ldots,a_{n-1}}   \sum\limits_{ A_1^n\leq a_n <2A_1^n} \left(\frac{1}{A_2^{n} q_{n}^{2}}\right)^{g_{(A_1 A_2),A_1}+2\varepsilon}\\
&\le \liminf_{N\to\infty}   \sum\limits_{n=N}^\infty   \sum\limits_{a_1,\ldots,a_{n-1}}   A_1^n \left(\frac{1}{(A_2A_1^2)^{n} q_{n-1}^{2}}\right)^{g_{(A_1 A_2),A_1}+2\varepsilon}\\
&\le \liminf_{N\to\infty}   \sum\limits_{n=N}^\infty  \frac{1}{2^{(n-1)\varepsilon}}<\infty. 
\end{split}
\end{equation*}
The third inequality holds by recursive relation (\ref{recu}). Thus, the upper bound  is obtained immediately by
$$
\dim_H E(A_1,A_2) \leq   g_{(A_1 A_2),A_1}.
$$
Hence, 
$$
\dim_H E(A_1,A_2) \leq \min \left\{ s_{A_1} , g_{(A_1 A_2),A_1} \right \}.
$$
\subsection{Lower bound.} In this subsection we will determine the lower bound for the dimension of $E(A_1,A_2)$ by using the mass distribution principle (Proposition \ref{p1}).


For convenience, let us define some dimensional numbers in first. For any integers $N,M$, define the dimensional number $s=s_N(M)$ and $g=g_N(M)$ repectively as the solution to
\begin{equation*}\label{def s_n}
\sum\limits_{1\leq a_1,\ldots,a_{N}\leq M} \left( \frac{1}{A_1^{N} q_{N}^{2}}\right)^s=1
\end{equation*}
and
\begin{equation*}\label{def g_n}
\sum\limits_{1\leq a_1,\ldots,a_{N}\leq M} \frac{A_1^N}{((A_1^2 A_2)^{N} q_{N}^{2})^g}=1.
\end{equation*}
More specifically, each equation has a unique solution and, by Corollary \ref{cor2.1}, 
$$
\lim\limits_{M\to\infty} \lim\limits_{N\to\infty} s_N(M) = s_{A_1},
$$
and
$$
\lim\limits_{M\to\infty} \lim\limits_{N\to\infty} g_N(M) = g_{(A_1 A_2),A_1}.
$$
Take a sequence of large sparse integers $\{\ell_k\}_{k\geq1}$, say, $\ell_k \gg e^{\ell_1+\dots+\ell_{k-1}}.$ For any $\varepsilon>0$, choose integers $N,M$ sufficiently large such that
$$
s> s_{A_1}-\varepsilon, \quad g>g_{(A_1 A_2),A_1}-\varepsilon, \qquad \left(2^{(N-1)/2})\right)^{\varepsilon/2}\geq 2^{100}.
$$
Without loss of generality, we assume that $s<g$. Let $$n_k-n_{k-1}=\ell_kN+1, \, \forall k\geq1,$$
such that
$$\left( 2^{\ell_k(N-1)/2} \right)^{\frac{\varepsilon}{2}} \ge  \prod\limits_{t=1}^{k-1} (M+1)^{\ell_t N}(A_1A_2)^{\sum_{i=1}^{t}\ell_{i}N+t}.$$
At this point, define a subset of $E(A_1,A_2)$ as
\begin{align}
E =  \{ x\in[0,1): \,  A_1^{n_k}\leq a_{n_k}(x) <2A_1^{n_k}, \  &  A_2^{n_k} \leq a_{n_k+1}<2  A_2^{n_k} \text{ for all } k\geq1 \label{mainsubset}\\
& \text{ and } a_n(x)\in\{1,\ldots,M\} \text{ for other } n\in\N \}.\notag
\end{align}

Next we proceed to make use of a symbolic space.  Define $D_0=\{\emptyset\}$, and for any $n\geq1$, define 
 \begin{align*}
    D_n=\Bigg\{(a_1,\cdots, a_n)\in \N^n: A_i^{n_k}\le a_{n_k+i}&<2A_i^{n_k}, \ {\text{for all}} \ 0\le i\le 1, k\ge 1 \ {\text{with}} \ {n_k+i}\le n;\\ &{\text{and}}\ a_j\in \{1,\cdots, M\}, \ {\text{for other $j\le n$}}\Bigg\}.
  \end{align*} This set is just the collection of the prefixes of the points in $E$.
Moreover, the collection of finite words of length $N$ is denoted by
$$
\UU = \{ w= (\sigma_1,\dots, \sigma_N): 1\leq \sigma_i \leq M, 1\leq i\leq N\}
$$
and, in the rest of the paper we always use $w$ to denote a generic word in $\UU$. 

\subsubsection{Cantor structure of $E$.}
In this subsection, we depict the structure of $E$ with the help of symbolic space as mentioned above. For any $(a_1,\cdots, a_n)\in D_n$, define $$
J_n(a_1,\cdots,a_n)=\bigcup_{a_{n+1}: (a_1,\cdots,a_n, a_{n+1})\in D_{n+1}}I_{n+1}(a_1,\cdots,a_n, a_{n+1})
$$ and call it a {\em basic cylinder} of order $n$. More precisely, for any $k\ge 0$\begin{itemize}

\item when $n_k+1\le n<n_{k+1}-1$ (by viewing $n_0=0$), $$
J_n(a_1,\cdots,a_n)=\bigcup_{1\le a_{n+1}\le M}I_{n+1}(a_1,\cdots,a_n, a_{n+1}).
$$
\item when $n=n_{k+1}-1$  and $n=n_{k+1}$, for $i=1,2$,  $$
J_n(a_1,\cdots,a_n)=\bigcup_{A_i^{n_{k+1}}\le a_{n+1}< 2 A_i^{n_{k+1}}}I_{n+1}(a_1,\cdots,a_n, a_{n+1}).
$$
\end{itemize}
Then we define the level $n$ of the Cantor set $E$ as
$$
\mathcal{F}_n=\bigcup_{(a_1,\cdots,a_n)\in D_n}J_n(a_1,\cdots,a_n).
$$ 
Consequently, the Cantor structure of $E$ is described as follows
$$
E=\bigcap_{n=1}^{\infty}\mathcal{F}_n=\bigcap_{n=1}^{\infty}\bigcup_{(a_1,\cdots,a_n)\in D_n}J_n(a_1,\cdots,a_n).
$$

We observe that every element $x\in E$ can be written as \begin{align*}
x=[w_1^{(1)},\cdots, w_{\ell_1}^{(1)}, a_{n_1}, a_{n_1+1}, & w_1^{(2)},\cdots, w_{\ell_2}^{(2)}, a_{n_2},a_{n_2+1},
 \cdots,  w_1^{(k)},\cdots, w_{\ell_k}^{(k)}, a_{n_k},a_{n_k+1},\cdots],
\end{align*} where $$
w^{(p)}_k\in \UU, \ {\text{and}}\ \ A_{i+1}^{n_k}\le a_{n_k+i}\le 2A_{i+1}^{n_k}, \ \text{ for all }  k,p\in\mathbb{N} \ \ 0\le i\le 1.
$$
Then the length of cylinder set can be estimated as follows.
\begin{lemma}[Length estimation] Let $x\in E$ and $n_k+1 \leq n < n_{k+1}-1$.
\begin{itemize}

\item for $n=n_{k}-1$,  \begin{align}\label{ff8}|J_{n_{k}-1}(x)| & \ge \frac{1}{2^3A_1^{n_{k}}}\cdot \left(\frac{1}{2^{\ell_{k}}}\cdot \prod_{i=1}^{\ell_{k}}\frac{1}{q_N(w_i^{(k)})}\cdot \frac{1}{q_{n_{k-1}+1}}\right)^2  
 \ge  \frac{1}{A_1^{n_{k}}} \left( \prod_{i=1}^{\ell_{k}}\frac{1}{q_N(w_i^{(k)})}\right)^{2(1+\varepsilon)}.
\end{align}
    
\item for $n=n_{k}$,
\begin{align}\label{ff9}|J_{n_{k}}(x)|\ge \frac{1}{2^{11}}\cdot \frac{1}{A_1^{n_{k}}A_2^{n_{k}}}\cdot |J_{n_{k}-1}(x)|.\end{align}
\item for $n=n_{k}+1$, \begin{align}\label{ff10}|J_{n_{k}+1}(x)|\ge \frac{1}{2^{11}}\cdot \frac{1}{A_1^{n_{k}}A_2^{2n_{k}}}\cdot |J_{n_{k}-1}(x)|.
\end{align}

  \item For each $1\le \ell<\ell_{k+1}$,
 \begin{align}\label{ff12}|J_{n_k+1+\ell N}(x)|&\ge \frac{1}{2^3}\cdot \left(\frac{1}{2^{2\ell}}\cdot \prod_{i=1}^{\ell}\frac{1}{q_N^2(w_i^{(k+1)})}\right)\cdot \frac{1}{q^2_{{n_k}+1}} \ge  \left( \prod_{i=1}^{\ell}\frac{1}{q_N^2(w_i^{(k+1)})}\right)^{1+\varepsilon}\cdot \frac{1}{q^2_{{n_k}+1}}. \end{align}

\item for $n_k+1+(\ell-1)N<n<n_k+1+\ell N$ with $1\le \ell\le \ell_{k+1}$, \begin{align}\label{ff13}
|J_{n}(x)|\ge c\cdot |J_{n_k+1+(\ell-1)N}(x)|,
\end{align} where $c=c(M, N)$ is an absolute constant.
\end{itemize}
\end{lemma}
\begin{proof}
The proof of this lemma is based on the following inequality. Let $$V_{a}:=\max{a_{n+1}}, \qquad v_{a}:=\min{a_{n+1}}.$$ Then
\begin{align*}
|J_n(x)|& = \left| \frac{v_{a}p_n+p_{n-1}}{v_{a}q_n+q_{n-1}}-\frac{(V_{a}+1)p_n+p_{n-1}}{(V_{a}+1)q_n+q_{n-1}}\right|=\frac{V_{a}+1-v_{a}}{(v_{a}q_n+q_{n-1})((V_{a}+1)q_n+q_{n-1})},
\end{align*}
$$\Longrightarrow \frac{1}{2^{3}}\cdot\frac{V_{a}-v_{a}}{V_{a}v_{a}q_n^2}\le|J_n(x)|\le\frac{V_{a}+1-v_{a}}{V_{a}v_{a}q_n^2}.$$
The cases of $n=n_{k}-1$, $n=n_{k}$ and any others follows by replacing $V_{a}$ with $2A_{1}^{n_{k}}$, $2A_{2}^{n_{k}}$, $M$ and $v_{a}$ with $A_{1}^{n_{k}}$, $A_{2}^{n_{k}}$, $1$ respectively. Then using subtle scale by the choice of $\ell_{k}$ is enough. More precisely, we give the main terms here and omit straightforward substitutions. 
\begin{align*}
\frac{1}{A_1^{n_k} q_{n_k-1}^2}  > &|J_{n_k-1}(x)|  \geq \frac{1}{8} \frac{1}{A_1^{n_k} q_{n_k-1}^2},\\
\frac{1}{(A_2 A^2_1)^{n_k} q_{n_k-1}^2} \geq \frac{1}{A_2^{n_k} q_{n_k}^2}  > &|J_{n_k}(x)|  \geq \frac{1}{8} \frac{1}{A_2^{n_k} q_{n_k}^2}\geq \frac{1}{2^7}  \frac{1}{(A_2 A^2_1)^{n_k} q_{n_k-1}^2}.
\end{align*}
For $n\neq n_{k}-1,n_{k}$,
$$\frac{1}{q_{n}^{2}}>|J_n(x)|\geq \frac{1}{8q_n^2}.$$
\end{proof}

\subsection{Mass distribution}\
In this subsection, we define two mass distributions along the basic intervals $J_n(x)$ containing $x$. These mass distributions then can be extended respectively into probability measure supported on $E$ by the Carath\'eodory extension theorem. Let us begin this idea by induction. For $n\le n_1+1$, 
\begin{enumerate}
  \item when $n=\ell N$ for each $1\le \ell\le \ell_1$, define $$
  \mu_1(J_{\ell N}(x))=\prod_{i=1}^{\ell}\frac{1}{q_N(w_i^{(1)})^{2s}\cdot A_1^{sN}}$$
and
$$
  \mu_2(J_{\ell N}(x))=\prod_{i=1}^{\ell}\frac{A_1^N}{q_N(w_i^{(1)})^{2g}\cdot  (A_2A_1^2)^{gN}}.
  $$
  We note that measures $\mu_1$ and $\mu_2$ can be defined on all basic cylinders of order $\ell N$ since  $x$ is arbitrary.  
  \item when $(\ell-1)N<n<\ell N$ for some $1\le \ell\le \ell_1$ and for all $1\le j\le 2$, define
  $$
  \mu_j(J_n(x))=\sum_{J_{\ell N}\subset J_n(x)}\mu_j(J_{\ell N}(x))
  $$ The consistency property as mentioned above fulfills the measure of other basic intervals of order less than $n_{1}-1$.
        \item when $n=n_{1}+i$ for each $0\le i\le 1$ and $1\le j\le 2$, define $$
  \mu_j(J_{n_1+i}(x))=\prod_{k=0}^i\frac{1}{A_k^n}\cdot \mu_j(J_{n_1-1}(x)).
  $$
\end{enumerate}

Assume the measure of all basic intervals of order $n_{k}+1$ has been defined when $n_{k}+1<n\le n_{k+1}+1$. 
\begin{enumerate}
  \item  When $n=n_{k}+1+\ell N$ for each $1\le \ell\le \ell_{k+1}$, define \begin{align}\label{ff7}
  \mu_1(J_{n_{k}+1+\ell N}(x))&=\prod_{i=1}^{\ell}\frac{1}{q_N(w_i^{(k+1)})^{2s}\cdot A_1^{sN}}\cdot \mu_1(J_{n_{k}+1}(x))
  \end{align}
and
\begin{align}\label{ff777}
  \mu_2(J_{n_{k}+1+\ell N}(x))=\prod_{i=1}^{\ell}\frac{A_1^N}{q_N(w_i^{(k+1)})^{2g}\cdot  (A_2A_1^2)^{gN}}\mu_2(J_{n_{k}+1}(x)) .
  \end{align}
  
  \item When $n_k+1+(\ell-1)N<n<n_k+1+\ell N$ for some $1\le \ell\le \ell_1$ and for $1\le j\le 2$, define
  $$
  \mu_j(J_n(x))=\sum_{J_{n_k+1+\ell N}\subset J_n(x)}\mu_j(J_{n_k+1+\ell N}).
  $$
Furthermore, for each measure, compared with the measure of a basic cylinder of order $n_k+1+(\ell-1) N$ and its offsprings of order $n_k+1+\ell N$, there is only a multiplier between them. More precisely, that is the term $$
\frac{1}{q_N^{2s}(w_{\ell}^{(k+1)})A_1^{sN}}
$$ in the case of $\mu_1$ and
$$
\frac{A_1^N}{q_N^{2g}(w_{\ell}^{(k+1)})(A_2A_1^2)^{gN}}
$$ in the case of $\mu_2$. Thus for each $1\le j\le2$, there is an absolute constant $c>0$ such that
$$\mu_j(J_n(x))\ge c\cdot \mu_j\Big(J_{n_k+1+(\ell-1)N}(x)\Big)$$
since the above two terms are uniformly bounded.  
  \item when $n=n_{k}+1+i$ for each $0\le i\le 1$ and $1\le j\le 2$, define 
\begin{align}\label{ff11}
  \mu_j(J_{n_{k}+1+i}(x))&=\prod_{j=0}^i\frac{1}{A_j^{n_{k}}}\cdot \mu_j(J_{n_{k}-1}(x)).
  \end{align}

  \item As for other orders of measure, to ensure the consistency property, let their measure equal to the summation of the measure of their offsprings. Moreover, for each integer $n$, the relation between measures of a basic cylinder and its predecessor acts like the case $n_k+1+(\ell-1)N<n<n_k+1+\ell N$, for each $1\le j\le 2$, there is a constant $c>0$ such that 
\begin{align}\label{ff2}
\mu_j(J_{n+1}(x))\ge c\cdot \mu_j(J_n(x)).
\end{align}
\end{enumerate}

\subsection{H\"{o}lder exponent of $\mu$ for basic cylinders}\
We compare the measure and length of $J_{n}(x)$. However, different measures for different values of parameters are considered as the following three cases:

\begin{enumerate}[label=\roman*)]
\item $A_1 \ge (A_1A_2)^s$ .
\item $(A_1A_2)^g < A_1 < (A_1A_2)^s $.
\item $A_1 \le  (A_1A_2)^g$.
\end{enumerate}

We will use measure $\mu_1$ for the first and the second cases, while we use $\mu_2$ for the third case. 

\begin{enumerate}
\item When $n={n_{k}-1}$.
Recall (\ref{ff7}) and \eqref{ff8}, it follows that for $\mu_{1}$,
\begin{align*}
\mu_1(J_{n_{k}-1}(x))&\le \frac{1}{ A_1^{sn_k}} \prod_{i=1}^{\ell_k}\frac{1}{q_N(w_i^{(k)})^{2s}} \le |J_{n_{k}-1}(x) |^{\frac{s}{1+\varepsilon}} \le 
\left(\frac{1}{A_1^{n_k} q_{n_k-1}^2} \right)^{\frac{s}{1+\varepsilon}} \le  \left(\frac{1}{A_1^{n_k} q_{n_k-1}^2} \right)^\frac{{\min\{s,g\}}}{1+\varepsilon}  .\end{align*}
For the measure $\mu_2$ recall \eqref{ff777} and \eqref{ff8}. It follows that
\begin{align*}
\mu_2(J_{n_{k}-1}(x))\le  \frac{A_1^{n_k}}{(A_1^2 A_2)^{g n_k}} \prod_{i=1}^{\ell_k}\frac{1}{q_N(w_i^{(k)})^{2g}} &\le  \frac{1}{A_1^{g n_k}} \prod_{i=1}^{\ell_k}\frac{1}{q_N(w_i^{(k)})^{2g}} \\
&  \le|J_{n_{k}-1}(x) |^{\frac{g}{1+\varepsilon}}  \le|J_{n_{k}-1}(x) |^\frac{{\min\{s,g\}}}{1+\varepsilon}.
\end{align*}

\item When $n={n_{k}}$.
Recall (\ref{ff11}) and \eqref{ff9}, for $\mu_{1}$ we get 
\begin{align*}
\mu_1(J_{n_{k}}(x))= \frac{1}{A_1^{n_{k}}}\cdot \mu_1(J_{n_{k}-1}(x)) &  \le  \frac{1}{A_1^{n_{k}}}\cdot \left(\frac{1}{A_1^{n_k} q_{n_k-1}^2} \right)^\frac{{\min\{s,g\}}}{1+\varepsilon}  \le  \left(  \frac{1}{(A_2 A^2_1)^{n_k} q_{n_k-1}^2}\right)^\frac{{\min\{s,g\}}}{1+\varepsilon}  \\
 &  \le c |J_{n_k}(x) |^\frac{{\min\{s,g\}}}{1+\varepsilon}  \le c \cdot  \left( \frac{1}{A_2^{n_k} q_{n_k}^2} \right)^\frac{{\min\{s,g\}}}{1+\varepsilon}.
\end{align*} 
For the measure $\mu_2$ we have
\begin{align*}
\mu_2(J_{n_{k}}(x))= \frac{1}{A_1^{n_{k}}}\cdot \mu_2(J_{n_{k}-1}(x)) &  \le  \frac{1}{A_1^{n_{k}}}\cdot  \frac{A_1^{n_k}}{(A_1^2 A_2)^{g n_k}} \prod_{i=1}^{\ell_k}\frac{1}{q_N(w_i^{(k)})^{2g}} \le  \left(  \frac{1}{(A_2 A^2_1)^{n_k} q_{n_k-1}^2}\right)^\frac{g}{1+\varepsilon}  \\
 &  \le c |J_{n_k}(x) |^\frac{g}{1+\varepsilon}  \le c \cdot  \left( \frac{1}{A_2^{n_k} q_{n_k}^2} \right)^\frac{g}{1+\varepsilon} \le  c \cdot  \left( \frac{1}{A_2^{n_k} q_{n_k}^2} \right)^\frac{{\min\{s,g\}}}{1+\varepsilon}  .
\end{align*}

\item When $n=n_{k}+1$. Recall \eqref{ff11} and \eqref{ff10}. Note that $0\le \min\{s,g\} \le 1$, for each $1\le j\le 2$,
\begin{align*}
  \mu_j(J_{n_{k}+1}(x))=\frac{1}{A_2^{n_{k}}}\cdot \mu_j(J_{n_{k}}(x))&\le \frac{1}{A_2^{n_{k}}} \cdot c \cdot  \left( \frac{1}{A_2^{n_k} q_{n_k}^2} \right)^\frac{{\min\{s,g\}}}{1+\varepsilon}\le c  \left( \frac{1}{A_2^{2n_k} q_{n_k}^2} \right)^\frac{{\min\{s,g\}}}{1+\varepsilon} \\ &\le c_2 |J_{n_k+1} (x)|^\frac{{\min\{s,g\}}}{1+\varepsilon} \le c_2  \left( \frac{1}{ q_{n_k+1}^2} \right)^\frac{{\min\{s,g\}}}{1+\varepsilon}.
\end{align*}

\item When $n=n_k+1+\ell N$ for some $1\le \ell<\ell_{k+1}$. Recall (\ref{ff7}), \eqref{ff777} and \eqref{ff12}, for each $1\le j\le 2$,
\begin{align*}
\mu_j(J_{n_k+1+\ell N}(x))
  \le  c_2   \prod_{i=1}^{\ell}\frac{1}{q_N(w_i^{(k+1)})^{2s}}  \left( \frac{1}{ q_{n_k+1}^2} \right)^\frac{{\min\{s,g\}}}{1+\varepsilon}\le c_2 | J_{n_k+1+\ell N}(x)|^\frac{{\min\{s,g\}}}{1+\varepsilon}. 
\end{align*}

\item For other $n$, let $1\le \ell\le \ell_k$ be the integer such that $$
n_k+1+(\ell-1)N\le n<n_k+1+\ell N.
$$
Recall \eqref{ff13}. Then for each $1\le j\le 2$,
\begin{align*}
  \mu_j(J_n(x))&\le \mu_j(J_{n_k+1+(\ell-1)N}(x)) \le c_2 \cdot \big|J_{n_k+2+(\ell-1)N}(x)\big|^\frac{{\min\{s,g\}}}{1+\varepsilon}\le c_2 \cdot c \cdot \big|J_{n}(x)\big|^\frac{{\min\{s,g\}}}{1+\varepsilon}.
\end{align*}
\end{enumerate}

In a summary, we have shown that for some absolute constant $c_3$, for any $n\ge 1$ and $x\in E$, \begin{align}\label{g1}
  \mu_j(J_n(x))\le c_3\cdot |J_n(x)|^\frac{{\min\{s,g\}}}{1+\varepsilon}.
\end{align}

\subsection{H\"{o}lder exponent for a general ball}

For simplicity, write
$$
\tau= \frac{{\min\{s,g\}}}{1+\varepsilon}.
$$
The next lemma gives a minimum gap between two adjacent fundamental cylinders.
\begin{lemma}[Gap estimation]\label{l3} Denote by $G_n(a_1,\cdots, a_n)$ the gap between $J_n(a_1,\cdots, a_n)$ and other basic cylinders of order $n$. Then $$
G_n(a_1,\cdots, a_n)\ge \frac{1}{M}\cdot |J_n(a_1,\cdots,a_n)|.
$$
\end{lemma}
\begin{proof}
  The proof of this lemma is derived from the positions of the cylinders in Proposition \ref{pp2}. We omit the details and refer the reader to its analogous proof in \cite[Lemma 5.3]{HuWuXu}.
\end{proof}

Then for any $x\in E$ and $r$ small enough, there is a unique integer $n$ such that 
$$G_{n+1}(x)\le r<G_{n}(x).$$ 
This implies that the ball $B(x,r)$ can only intersect one basic cylinder $J_n(x)$, and so all the basic cylinders of order $n+1$ for which $B(x,r)$ can intersect are all contained in $J_n(x)$. We discuss the measure of the ball with several cases. Note that $n_{k-1}+1\le n <n_{k}+1$. 
\begin{enumerate}
\item For $n_{k-1}+1\le n<n_{k}-1$, by (\ref{ff2}) and (\ref{g1}), it follows that for each $1\le j\le 2$,
\begin{align*}
\mu_j(B(x,r))&\le \mu_j(J_n(x))\le c\cdot \mu_j(J_{n+1}(x))\le c\cdot c_3\cdot \big|J_{n+1}(x)\big|^{\tau}\\&\le c\cdot c_3\cdot M\cdot (G_{n+1}(x))^{\tau}\le c\cdot c_3\cdot M\cdot  r^{\tau}.
\end{align*}

\item For $n=n_{k}-1$, the ball $B(x,r)$ can only intersect one basic cylinder $J_{n_k-1}(x)$ of order $n_k-1$. Next, the number of basic cylinders of order $n_{k}$ which are contained in $J_{n_k-1}(x)$ and intersect the ball can be calculated as follows.
 
We write $x=[a_1(x),a_2(x),\ldots]$ and observe that any basic cylinder $J_{n_{k}}(a_{1},\ldots,a_{n_k})$ is contained in the cylinder $I_{n_k}(a_{1},\ldots,a_{n_k})$. Note that $A_{1}^{n_{k}}\le a_{n_{k}}\le 2A_{1}^{n_{k}}$, the length of cylinder $I_{n_k}$ is
$$
\frac{1}{q_{n_k}(q_{n_k}+q_{n_k-1})}\ge \frac{1}{2^{5}}\cdot \frac{1}{q_{n_k-1}^2A_1^{2n_k}}.
$$
We also note that radius $r$ is sometimes too small to cover a whole cylinder of order $n_{k}$. The exposition needs to split into two parts. When 
$$r<\frac{1}{2^{5}} \frac{1}{q_{n_k-1}^2(u)A_1^{2n_k}},$$ 
then the ball $B(x,r)$ can intersect at most three cylinders $I_{n_k}(a_1,\ldots,a_{n_{k}})$ and so three basic cylinders $J_{n_k}(a_1,\ldots,a_{n_{k}})$. Since each measure has the same distribution on these intervals, for $1\le j\le 2$,
\begin{align*}
  \mu_j(B(x,r))&\le 3\mu_j(J_{n_k}(x))\le 3 \cdot c_3\cdot |J_{n_k}(x)|^{\tau}\\ 
&\le 3\cdot c_3\cdot M\cdot G_{n+1}(x)^{\tau}\le 3\cdot c_1\cdot M\cdot r^{\tau}.
  \end{align*}

When $$
  r\ge \frac{1}{2^{5}} \frac{1}{q_{n_k-1}^2A_1^{2n_k}},
  $$ then the number of basic cylinders of order $n_k$ for which the ball $B(x,r)$ can intersect is at most $$
  {2^{6}r}\cdot q_{n_k-1}^{2}A_1^{2n_{k}}+2\le 2^7\cdot {r}\cdot q_{n_k-1}^{2}A_1^{2n_{k}}.$$ 
Thus, for $1\le j\le 2$,
 \begin{align*}
    \mu_j(B(x,r))&\le \min\Big\{\mu_j(J_{n_k-1}(x)),\ \  2^7\cdot {r}\cdot q_{n_k-1}^{2}(u)A_1^{2n_{k}}\cdot \frac{1}{A_1^{n_k}}\cdot \mu_j(J_{n_k-1}(x))\Big\}\\
    &\le c_3\cdot |J_{n_k-1}|^{\tau}\cdot \min\Big\{1, 2^7\cdot {r}\cdot q_{n_k-1}^{2}(u)A_1^{n_{k}}\Big\}\\
    &\le c_3\cdot \left(\frac{1}{q_{n_k-1}(u)^2 A_1^{n_k}}\right)^{\tau}\cdot 1^{1-\tau}\cdot \Big(2^7\cdot {r}\cdot q_{n_k-1}^{2}(u)A_1^{n_{k}}\Big)^{\tau(1-\epsilon)}\\
    &=c_4 \cdot r^{\tau}.
  \end{align*}

\item When $n=n_k$. Replaced $n_k-1$ and $A_1$ in case (2) by $n_k$ and $A_2$ respectively can arrive the same conclusion with the same argument as in case (2).
\end{enumerate}
\subsection{Conclusion}
Thus by applying the mass distribution principle (Proposition \ref{p1}), it yields that
\begin{equation*}\label{dim min}
\dim_H E(A_1,A_2) \ge  \frac{{\min\{s,g\}}}{1+\varepsilon}.
\end{equation*}

Since $\varepsilon>0$ is arbitrary, by letting $N\to\infty$ as then $M\to\infty$, we arrive at
$$
\dim_H E(A_1,A_2) \geq \min \{s_{A_1}, \ g_{(A_1 A_2),A_1} \}.
$$
This finishes the proof.

\section{ Hausdorff dimension of $\FF_{B_1,B_2}$ }\label{specificsec}
\noindent
In this section we prove Theorem \ref{specific}. There are different possible cases for $B_1$ and $B_2$.
\subsection{ Case $B_1^{s_0} \le B_2$}
To get the upper bound for dimension of $\FF_{B_1,B_2}$ recall that for every $B_2$ 
$$
\FF_{B_1,B_2} \subset \mathcal{E}_2 = \{ x: a_n(x) a_{n+1}(x) \ge B_1^n \text{\,\, for i.m. \,\,} n\in\N \}.
$$
Hence, by the Theorem \ref{HWXthm}, we have  that
$$
\dim_H \FF_{B_1,B_2} \le s_0.
$$
To get a lower bound, recall a set \eqref{mainsubset} from Theorem \ref{dim E(A_1,A_2)} and let $A_1=B_1^{s_0}, A_1 A_2 = B_1$, so that we have a set

\begin{equation*}
E=\left\{x\in[0,1):
\begin{split} 
 &\begin{split}&c_1 B_1^{ns_0}\leq a_n(x) <2c_1 B_1^{ns_0}, \\ &c_2 B_1^{n(1-s_0)} \leq a_{n+1}(x)<2  c_2 B_1^{n(1-s_0)},\end{split}&\text{ for i.m. } n\in\N \\
&1\leq a_n(x) \leq M &\text{for all other\,\,} n\in\N
\end{split}
\right\}.
\end{equation*}
Since we have  $B_1^{s_0} \le B_2$ and $s_{0}\ge 1/2$ by Propositon 2.11 in \cite{HuWuXu}, one can easily see that $E$ is indeed a subset of $\FF_{B_1,B_2}$ and by the proof of Theorem \ref{dim E(A_1,A_2)} we know that 
$$
\dim_H E \ge  \min \{s_{B_1^{s_0}}, \ g_{B_1,B_1^{s_0}} \}  =  s_{B_1^{s_0}}  = s_0,
$$
so $\dim_H \FF_{B_1,B_2} \ge s_0.$
Hence, we get that when $B_1^{s_0} \le B_2$, we have $\dim_H \FF_{B_1,B_2} =s_0.$

\subsection{ Case  $B_1^{s_0} \ge B_2 > B_1^{1/2}$ }\label{subsection42}
For the upper bound a subset
$$
U_1 = \left\{ x: 1 \le a_n(x) \le B_2^n, \ a_{n+1}(x) \ge \frac{B_1^n}{a_n(x)} \text{\,\, for i.m. \,}  n\in\N \right\}
$$
of $  \FF_{B_1,B_2}.$ This set can be covered by collections of basic cylinders $J_n$ of order $n$:
\begin{align*}
U_1 &= \bigcup_{a_1,\ldots,a_{n-1}\in\N} \left\{ x\in[0,1) :a_k(x)=a_k, 1\le k \le n-1,1 \le a_n(x) \le B_2^n,  a_{n+1}(x) \ge \frac{B_1^n}{a_n(x)} \right\} \\
&  =  \bigcup_{a_1,\ldots,a_{n-1}\in\N}  \bigcup_{1 \le a_n \le B_2^n}  \bigcup_{ a_{n+1} \ge \frac{B_1^n}{a_n}} I_{n+1}(a_1,\ldots,a_{n+1}) \\
& =   \bigcup_{\substack {a_1,\ldots,a_{n-1}\in\N \\ 1 \le a_n \le B_2^n }}  J_n(a_1,\ldots,a_n).
\end{align*}
Note that 
$$
J_n(a_1,\ldots,a_n) = \bigcup_{ a_{n+1} \ge \frac{B_1^n}{a_n}} I_{n+1}(a_1,\ldots,a_{n+1}),
$$
and
$$
|J_n(a_1,\ldots,a_n)| \asymp \frac{1}{B_1^n a_n q_{n-1}^2}.
$$
Now consider the $s$-volume of the cover of $U_{1}$:
$$
\sum\limits_{a_1,\ldots,a_{n-1}\in\N}  \sum\limits_{1 \le a_n \le B_2^n}  \left(  \frac{1}{B_1^n a_n q_{n-1}^2} \right)^s \asymp
 \sum\limits_{a_1,\ldots,a_{n-1}\in\N} B^{n(1-s)}_2 \left(  \frac{1}{B_1^n  q_{n-1}^2} \right)^s.
$$
Therefore, an $s$-dimensional Hausdorff measure of $U_1$ can be estimated as
$$
\HH^s(E(A_1,A_2)) \leq  \liminf_{N\to\infty}   \sum\limits_{n=N}^\infty    \sum\limits_{a_1,\ldots,a_{n-1}\in\N} B^{n(1-s)}_2 \left(  \frac{1}{B_1^n  q_{n-1}^2} \right)^s.
$$
Hence, we have that
$$
\dim_H  \FF_{B_1,B_2} \le \dim_H U_1 \le  g_{B_1,B_2}.
$$
To get a lower bound, recall a set \eqref{mainsubset} from Theorem \ref{dim E(A_1,A_2)} and let $A_1=B_2,  \ A_1 A_2=B_1$, so that we have 
\begin{equation*}
E=\left\{x\in[0,1):
\begin{split} 
 &\begin{split}&c_1 B_2^{n}\leq a_n(x) <2c_1 B_2^{n}, \\ & c_2\frac{B_1^n}{B_2^n} \leq a_{n+1}(x)<2  c_2\frac{B_1^n}{B_2^n},\end{split}&\text{ for i.m. } n\in\N \\
&1\leq a_n(x) \leq M &\text{for all other\,\,} n\in\N
\end{split}
\right\}.
\end{equation*}
Notice that if $B_2<B_1^{s_0}$, we have 
$$
\min \{s_{B_2},g_{B_1,B_2} \}  = g_{B_1,B_2}.
$$
So in this case we have
$$
\dim_H E \ge g_{B_1,B_2}.
$$
Since $ B_2 > B_1^{1/2}$, or $B_2^2>B_1$, one can easily see that $E$ is indeed a subset of $\dim_H  \FF_{B_1,B_2}$.
Hence
$$
\dim_H  \FF_{B_1,B_2} \ge g_{B_1,B_2}.
$$
Hence, together with the upper bound,  when $B_1^{s_0} \ge B_2 > B_1^{1/2}$, we have that $\dim_H  \FF_{B_1,B_2} = g_{B_1,B_2}$.

\subsection{ Case $B_1^{1/2}\ge B_2$}
In this case let us investigate the conditions on partial quotients in the definition of $ \FF_{B_1,B_2}$. We know that $a_{n+1}<B_2^n$ and $a_{n}<B_2^{n-1}$.\\
Hence we get 
$$
a_n a_{n+1} <B_2^{2n-1}.
$$
But $B_1^{1/2}\ge B_2$, we come to a contradiction with $a_n a_{n+1} \ge B_1^n$. Hence in this case the set $ \FF_{B_1,B_2}$ is empty.

\section{Hausdorff dimension of  $\FF(\Phi_1,\Phi_2)$}\label{FPhi1Phi2}
To make it easier to follow, let us recall the statement of the result.
Let us consider arbitrary positive functions $\Phi_1$ and $\Phi_2$. Suppose

\begin{align*}
\log B_1 &= \liminf\limits_{n\to\infty} \frac{\log \Phi_1(n)}{n}  \quad \text{and} \quad  \log b_1 = \liminf\limits_{n\to\infty} \frac{\log\log \Phi_1(n)}{n}.\\
\log B_2 &= \lim\limits_{n\to\infty} \frac{\log \Phi_2(n)}{n}   \quad \ \ \text{and}\quad   \log b_2 = \lim\limits_{n\to\infty} \frac{\log\log \Phi_2(n)}{n}.
\end{align*}
Then


 \begin{enumerate}
\item  if $B_1^{s_0} \le B_2$ and $B_1$ is finite, then  $$    \dim_H \FF(\Phi_1,\Phi_2) = s_0.     $$ 
\item  if $B_1^{s_0} \ge B_2 > B_1^{1/2}$, then $$ \dim_H \FF(\Phi_1,\Phi_2) = g_{B_1,B_2}.$$
\item  if $B_1^{1/2}> B_2$ and $B_2$ is finite, then $$  \FF(\Phi_1,\Phi_2) = \emptyset .$$
\item  if  $B_1= B_2=\infty$, $1\le b_1 < b_2$ and $b_1$ is finite, then $$  \dim_H \FF(\Phi_1,\Phi_2) = \frac{1}{1+b_1}.  $$
\item  if  $B_1= B_2=\infty$, $b_1> b_2\ge1$ and $b_2$ is finite, then $$  \FF(\Phi_1,\Phi_2) = \emptyset .$$
\item  if  $B_1= B_2=\infty$, $b_1=b_2=\infty$, then 
$$
\text{either\,\, }  \FF(\Phi_1,\Phi_2) = \emptyset  \text{\,\,  or\,\, }    \dim_H \FF(\Phi_1,\Phi_2) = 0. 
$$

\end{enumerate}

\begin{proof}
\,
\begin{enumerate}
\setlength\itemsep{2em}
\item $B_1^{s_0} \le B_2$ and $B_1$ is finite. \\
For any $\epsilon>0$, one has \
$$
\Phi_1(n) \ge (B_1 - \epsilon)^n, \text{\,\, for all \,} n \gg 1.
$$
Thus 
$$
\FF(\Phi_1,\Phi_2) \subset  \left\{  x\in [0,1) : a_{n}(x) a_{n+1}(x) \ge (B_1 - \epsilon)^n, \,\,\text{ i.m. n } \in\N \right\}.
$$
By Theorem \ref{HWXthm} and continuity of dimensional number established in \cite{HuWuXu}, we have
$$
 \dim_H \FF(\Phi_1,\Phi_2) \le s_{0}.
$$
As for the lower bound, fix $\widetilde{B_1} > B_1$ and $\widetilde{B_2} < B_2$, where we take finite $\widetilde{B_2}$ independently of whether $B_2$ is finite or not. Then by the definition of $B_1$, we can choose a largely sparse integer sequence $\{n_j\}_{j\ge1}$ such that for all $j\ge1$,
 
$$
\Phi_1(n_j) \le \widetilde{B_1}^{n_j}.
$$
 By the definition of $B_2$ one has 
$$
\Phi_2(n) \ge  \widetilde{B_2}^n, \text{\,\, for all \,} n \gg 1. 
$$
So we can see that
$$
\FF(\Phi_1,\Phi_2) \supset \left\{x\in[0,1):
\begin{split} 
 a_{n_j}(x)a_{n_j+1}(x) & \geq \widetilde{B_1}^{n_j} &\text{for all } j\in\N \\
 a_{n+1}(x) & < \widetilde{B_2}^n  &\text{for all sufficiently large } n\in\N
\end{split}
\right\}.
$$
In the definition of \eqref{mainsubset} we required that $n_{k+1}-n_k-2$ is a multiple of some large integer $N$. Here we cannot guarantee this. However, we can define a slightly different subset as $\widetilde{E}$.\\
Still fix a large integer $N$ and an integer $M$. Let 
$$
n_{k+1}-n_k-2=\ell_{k+1}N+i_{k+1}, \text{\,\, for some } 0 \le i_{k+1} < N.
$$
Denote
$$
y_k = n_{k-1} + 1 + \ell_k N.
$$


Now we consider the following set, which is the subset of our set under consideration:
\begin{align*}
\widetilde{E} = \Bigg\{ x\in [0,1) :\, & c_1\widetilde{B_1}^{s_0 n_j}  \le  a_{n_j}(x) < 2 c_1 \widetilde{B_1}^{s_0 n_j}, c_2 \left(\frac{\widetilde{B_1}}{\widetilde{B_1}^{s_0}} \right)^{n_j} \le a_{n_j+1}(x) < 2 c_2 \left(\frac{\widetilde{B_1}}{\widetilde{B_1}^{s_0}}\right)^{n_j},  \\
& \a_n(x) = 2 \text{\,\, for all } y_j<n<n_j, j\ge 1, \text{\,\, and } 1 \le a_n(x)   \le M \text{\,\, for other } n\in\N  \Bigg\}.
\end{align*}
The rest of the argument can be carried out with no more changes to conclude that 
$$
\hdim \FF(\Phi_1,\Phi_2) \ge s_{\widetilde{B_1}^{s_{0}}}.
$$
By the continuity of $s_{B}$ and using the upper bound, we have $\hdim \FF(\Phi_1,\Phi_2) = s_0$.

\item  $B_1^{s_0} \ge B_2 > B_1^{1/2}$.\\
To get the upper bound for dimension of $\FF(\Phi_1,\Phi_2)$ in this case, we can take the following set $U_2$, which is a superset of $\FF(\Phi_1,\Phi_2)$:
$$
U_2 = \left\{ x: 1 \le a_n(x) \le (B_2+\epsilon_2)^n, \  a_{n+1}(x) \ge \frac{(B_1-\epsilon_1)^n}{a_n(x)} \text{\,\, for i.m. \,\,}  n\in \N\right\}.
$$
Now one can repeat the argument from the subsection \eqref{subsection42}. By the continuity of $g_{B_1,B_2}$ with respect to $B_1$ and $B_2$ we get that
$$ 
\dim_H \FF(\Phi_1,\Phi_2) \le g_{B_1,B_2}.
$$
For the lower bound, we can completely repeat the argument from the lower bound in last case, but for a slightly different set, namely
\begin{align*}
\widetilde{E} = \Bigg\{ x\in [0,1) :& \widetilde{B_2}^{ n_j} c_1 \le  a_{n_j}(x) < 2 c_1 \widetilde{B_2}^{n_j}, c_2 \left(\frac{\widetilde{B_1}}{\widetilde{B_2}} \right)^{n_j} \le a_{n_j+1}(x) < 2 c_2 \left(\frac{\widetilde{B_1}}{\widetilde{B_2}}\right)^{n_j},  \\
& a_n(x) = 2 \text{\,\, for all } y_j<n<n_j, j\ge 1, \text{\,\, and } 1 \le a_n(x)   \le M \text{\,\, for other } n\in\N  \Bigg\}.
\end{align*}
By the continuity of  $g_{B_1,B_2}$ in corollary \ref{cor2.1} and using the upper bound, we have $\hdim \FF(\Phi_1,\Phi_2) =  g_{B_1,B_2}$.
\item $B_1^{1/2}> B_2$ and $B_2$ is finite.\\
 If $B_1 = +\infty$, then obviously our set is empty. Hence we can assume that $B_1 < \infty$. By definition of $B_1$ and $B_2$ we know that $\Phi_1(n)\ge (B_1-\varepsilon_1)^n$ for $n$ large enough with any $\varepsilon_1>0$ and $\Phi_2(n)\le(B_2+\varepsilon_2)^n$ for $n$ large enough with any $\varepsilon_2>0$.\\
By conditions of our set we know that 
$$
a_{n+1} < \Phi_2(n) \le(B_2+\varepsilon_2)^n \text{\,\,\, and \,\,\,}  a_n < \Phi_2(n) \le(B_2+\varepsilon_2)^{n-1}.
$$
Combining these together, we have that $a_n a_{n+1} < (B_2 + \varepsilon_2)^{2n-1}$. On the other hand, the second condition of our set is  
$$
a_{n+1}a_n \ge \Phi_1(n) \ge (B_1-\varepsilon_1)^n.
$$
In our case we know that $B_1>B_2^2$, or we can rewrite it as $B_1 = B^2_2 + \delta$ for some fixed $\delta>0$. Let $\delta' = \sqrt{B_2^2+\delta/2}-B_2$ and take $\varepsilon_1 < \delta/2$ and $\varepsilon_2 < \delta'$. Now we easily come to a contradiction between conditions, since
$$
 (B_1-\varepsilon_1)^n = (B_2^2+\delta - \varepsilon_1)^n >  (B_2^2+\delta/2)^n = (B_2 + \delta')^{2n} >  (B_2 + \varepsilon_2)^{2n-1},
$$
which contradicts our conditions. Hence our set is empty.

\item  $B_1= B_2=\infty$, $1\le b_1 < b_2$ and $b_1$ is finite.\\
First, let us assume that $b_1=1$. Note that $B=\infty$, therefore for any $C>1$, $\Phi_1(n) \ge C^n$ for all sufficiently large $n\in\N$. Thus one can see that
$$
  \FF(\Phi_1,\Phi_2) \subset  \{ x\in [0,1) : a_n(x) \ge  C^n \text{ \,\, i.m. } n\in\N \}.
$$
Then by Theorem \ref{WaWu} we have
$$
\hdim  \FF(\Phi_1,\Phi_2) \leq \lim_{C\to\infty} s_C = \frac{1}{2}.
$$
Now for the lower bound notice that $b_2=1+\delta$ for some fixed $\delta>0$ and that for every $0<\epsilon<\frac{\delta}{2}$ one has
$$
\Phi_1(n) \leq e^{(1+\epsilon)^n} \text{\,\, and \,\, } \Phi_2(n) \geq e^{(1+\delta-\epsilon)^n}.
$$
Now it is obvious that
$$
\FF(\Phi_1,\Phi_2) \supset \left\{x\in[0,1):
\begin{split} 
 a_{n}(x)a_{n+1}(x) & \geq e^{(1+\epsilon)^n} &\text{for i.m. } n\in\N \\
 a_{n+1}(x) & <e^{(1+\delta-\epsilon)^n}  &\text{for all sufficiently large } n\in\N
\end{split}
\right\}.
$$
At this point we will make use of Lemma \ref{helplemma}. Set $s_n = \frac{1}{2} e^{(1+\epsilon)^n}$. One can easily check that
$$
U_3 = \{ x\in[0,1) : s_n \le a_{n+1} < 2 s_n  \,\,\, \forall n\in\N \}
$$
is a subset of our original set.\\
Applying Lemma \ref{helplemma} to this choice of the sequence $s_n$, we get
\begin{align*}
\hdim  \FF(\Phi_1,\Phi_2)  &\geq   \lim_{\epsilon\to 0}\left( \liminf_{n\to\infty} \frac{\log(s_1 s_2\cdots s_n) }{2\log(s_1 s_2\cdots s_n)+\log s_{n+1}}\right) \\
&=  \lim_{\epsilon\to 0}\left( \liminf_{n\to\infty} \frac{1-\frac{1}{(1+\epsilon)^n} }{2+\epsilon-\frac{2}{(1+\epsilon)^n}}\right) \\ &=  \lim_{\epsilon\to 0} \frac{1}{2+\epsilon} = \frac{1}{2}.
\end{align*}
\\
Now let $b_1>1$. By the definition of $b_1$, for any $\epsilon>0, \frac{\log\log\Phi_1(n)}{n} \leq \log (b_1 + \epsilon)$ that is $\Phi_1(n) \leq e^{(b_1+\epsilon)^n}$ for infinitely many $n\in\N$, whereas $\Phi_1(n) \geq   e^{(b_1-\epsilon)^n}$ for all sufficiently large $n\in\N$.
Note that
$$
 \FF(\Phi_1,\Phi_2) \subset  \left\{ x\in [0,1) : a_n(x) \ge    e^{\frac{1}{2(b_1-\epsilon)}(b_1-\epsilon)^n} \text{ \,\, for  i.m. } n\in\N \right\}.
$$
Therefore, by using Lemma \ref{lemb} we get an upper bound
$$
\hdim \FF(\Phi_1,\Phi_2) \leq \lim_{\epsilon\to0} \frac{1}{1+b_1-\epsilon} = \frac{1}{1+b_1}.
$$
Now for the lower bound we use Lemma \ref{helplemma} once again. Set $s_n = \frac{1}{2} e^{(b_1+\epsilon)^n}$. One can easily check that
$$
U_4 = \{ x\in[0,1) : s_n \le a_{n+1}(x) < 2 s_n  \,\,\, \forall n\in\N \}
$$
is a subset of our original set. Applying Lemma \ref{helplemma} to this choice of the sequence $s_n$, we get
\begin{align*}
\hdim  \FF(\Phi_1,\Phi_2)  &\geq   \lim_{\epsilon\to 0}\left( \liminf_{n\to\infty} \frac{\log(s_1 s_2\cdots s_n) }{2\log(s_1 s_2\cdots s_n)+\log s_{n+1}}\right) \\
& = \lim_{\epsilon\to 0}\left( \liminf_{n\to\infty} \frac{1-\frac{1}{(b_1+\epsilon)^n} }{b_1+\epsilon+1-\frac{2}{(b_1+\epsilon)^n}}\right) \\ &=  \lim_{\epsilon\to 0} \frac{1}{1+b_1+\epsilon} \\ &= \frac{1}{1+b_1}.
\end{align*}

\item  if  $B_1= B_2=\infty$, $b_1 > b_2\ge 1$ and $b_2$ is finite.\\
We want to show that in this case our set is empty. Let us first assume that $b_1$ is finite. Then $b_1 = b_2 + \delta$ for some fixed $\delta>0$. Next, by the conditions of our set, for every $\epsilon_2>0$ one has
$$
a_{n} < \Phi_2(n) \le e^{(b_2+\epsilon_2)^{n-1}}, \  a_{n+1} < \Phi_2(n) \le e^{(b_2+\epsilon_2)^n}.
$$
Combining these together, we have that
$$
a_{n} a_{n+1} <  e^{(b_2+\epsilon_2)^{n-1}+(b_2+\epsilon_2)^n}.
$$
On the other hand, the second condition of our set is that for every $\epsilon_1>0$ one has
$$
a_n a_{n+1} \ge \Phi_1(n) \ge e^{(b_1-\epsilon_1)^{n}}  =  e^{(b_2+\delta-\epsilon_1)^{n}}.
$$
We can take  $\epsilon_1$ and $\epsilon_2$ small enough, so that
\begin{align*}
\log a_n a_{n+1} &\ge (b_2+\delta-\epsilon_1)^n > (b_2+\delta/2)^n   \\
& >2(b_2+\epsilon_2)^n > (b_2+\epsilon_2)^{n-1}+(b_2+\epsilon_2)^n > \log a_n a_{n+1}.
\end{align*}
This clearly gives us a contradiction for $n$ large enough, because $\delta$ is a fixed positive number. Hence our set is empty. Now if $b_1=+\infty$, then simply for some finite $b_1' > b_2^2 $ one has
$$
a_n a_{n+1} \ge \Phi_1(n) \ge e^{(b_1'-\epsilon_1)^{n}},
$$
and then proceed like for finite $b_1$ earlier.

\item $B_1= B_2=\infty$, $b_1=b_2=\infty$.  \\

From the definition of $b_1$, for every $C>1$ we have $\Phi_1(n) \ge e^{C^n} $. So we can consider the following superset of our set $ \FF(\Phi_1,\Phi_2)$:
$$
 \FF(\Phi_1,\Phi_2) \subset \left\{ x\in [0,1) : a_n(x) \ge  e^{C^n} \text{ \,\, i.m. } n\in\N \right\}.
$$
Then by Lemma \ref{lemb}
$$
\hdim \FF(\Phi_1,\Phi_2) \leq \lim_{C\to\infty} \frac{1}{C+1} = 0.
$$
So if the set is non-empty, then its dimension is 0. However, it can happen that under conditions of this case the set $ \FF(\Phi_1,\Phi_2)$ is empty.
Recall the conditions of our set,
$$ a_{n+1} < \Phi_2 (n),\, a_n < \Phi_2(n-1) \text{\,\, and \,\, } a_n a_{n+1} \ge \Phi_1 (n).$$
So if for all $n$ large enough
$$
\Phi_2(n) \Phi_2(n-1) \leq \Phi_1(n),
$$
then the set $ \FF(\Phi_1,\Phi_2)$ is empty because conditions are incompatible.

\end{enumerate}
\end{proof}

\section{Proofs of the propositions}\label{examples}

In this section we prove Propositions \ref{Prop1} and \ref{Prop2}.
\subsection{Proof of Proposition \ref{Prop1}} Note that this proposition concerns the case $B_1=B_2^2$. Recall that
$$
E_1 =\left\{x\in[0,1):
\begin{split} 
 a_n(x)a_{n+1}(x) & \geq B_2^{2n} \text{\,\, for infinitely many } n\in\N \\
 a_{n+1}(x) & < B_2^n  \text{\,\, for all sufficiently large } n\in\N
\end{split}
\right\}.
$$
Then simply by Theorem \ref{specific}, this set is empty.\\
Our second set is
$$
E_2 =\left\{x\in[0,1):
\begin{split} 
 a_n(x)a_{n+1}(x) & \geq B_2^{2n-1} \text{\,\, for infinitely many } n\in\N \\
 a_{n+1}(x) & <3 B_2^{n+1}  \text{\,\, for all sufficiently large } n\in\N
\end{split}
\right\}.
$$
To get the lower bound on the Hausdorff dimension, we take the subset
 \begin{align*}
F = \{ x\in[0,1)\, : \,  B_2^{n}\leq a_n(x) <2 B_2^{n},& \  B_2^n \leq a_{n+1}(x)<2  B_2^n,\,\, \text{i.m. } n\in\N; \\
& \text{ and } 1\leq a_n(x) \leq M \text{\,\, for all other\,\,} n\in\N \}.
\end{align*}
It is indeed a subset of $E_2$, because for the first condition of $E_2$ we have
$$
a_n (x)a_{n+1}(x) \ge B_2^n B_2^n = B_2^{2n} > B_2^{2n-1},
$$
and for the second one we have
$$
a_{n+1}(x) < 2B_2^n < 3B_2^{n+1},
$$
$$
a_n(x) < 2B_2^n < 3 B_2^n.
$$
We dealt with set $F$ in Theorem \ref{dim E(A_1,A_2)}, which tells us that its dimension is $g_{B_1,B_2}$. So, we have that
$$
\hdim E_2 \ge g_{B_1,B_2}.
$$
The upper bound of dimension for the set $E_2$ is clear by the result for the case $B_1^{1/2} < B_2$ in Theorem \ref{specific}, since $E_2$ is becoming larger when we decrease the value of $B_1$.
And so $\hdim E_2 = g_{B_1,B_2}.$
\subsection{Proof of Proposition \ref{Prop2}} Note that this proposition concerns the case $B_1 = B_2 = \infty$, $b_1=b_2<\infty$. The first set in consideration is $$
P_1 =\left\{x\in[0,1):
\begin{split} 
 a_n(x)a_{n+1}(x) & \geq e^{b_1^{n}} \text{\,\, for infinitely many } n\in\N \\
 a_{n+1}(x) & < e^{b_1^{n}}  \text{\,\, for all sufficiently large } n\in\N
\end{split}
\right\}.
$$
This set is  a particular case of the set $\FF (\Phi)$ from Theorem \ref{BBHthm} and its Hausdorff dimension by this theorem is equal to $\frac{1}{b_1+1}$.\\
The second set is
$$
P_2 =\left\{x\in[0,1):
\begin{split} 
 a_n(x)a_{n+1}(x) & \geq e^{5b_1^{n}} \text{\,\, for infinitely many } n\in\N \\
 a_{n+1}(x) & <  e^{b_1^{n-1}}  \text{\,\, for all sufficiently large } n\in\N
\end{split}
\right\}.
$$
This set is empty, because the  conditions are incompatible. Indeed, on the one hand we have
$$
a_n(x) a_{n+1}(x) <  e^{b_1^{n-1}+b_1^{n-2}}.
$$
On the other hand,
$$
a_n (x)a_{n+1}(x) \ge  e^{5b_1^{n}}.
$$
Combining this together, we get $5b_1^2 < b_1 +1$, which can happen only when $b_1$ is less than 1. But we know that $b_1 \ge 1$, hence the set $P_2$ is empty.

\end{document}